\documentclass[%
]{mpi2015-cscpreprint}

\usepackage{geometry}
\usepackage{amsmath}
\usepackage{amsthm}
\usepackage{amssymb}
\usepackage{graphics} % for pdf, bitmapped graphics files
\usepackage{epsfig} % for postscript graphics files
\usepackage{amsmath} % assumes amsmath package installed
\usepackage[textsize=tiny]{todonotes}
\usepackage{amssymb}  % assumes amsmath package installed
\usepackage{algorithmicx}
\usepackage{algorithm}
\usepackage{algpseudocode}
\usepackage{cite}%
\usepackage{bm}
\usepackage{subfigure}
\usepackage{tikz,pgfplots} 
\usepackage[normalem]{ulem}

\pgfplotsset{compat=newest} 
\pgfplotsset{plot coordinates/math parser=false} 
\usepackage{color}

\usepackage{mymacros}

\renewcommand{\hat}{\widehat}
\renewcommand{\tilde}{\widetilde}

\usepackage[square,sort&compress,comma,numbers]{natbib}% Citation support using natbib.sty
\bibpunct[, ]{[}{]}{,}{n}{,}{,}% Citation support using natbib.sty
\newcounter{mymac@matlab}
\newcommand{\matlab}{MATLAB% 
	\ifnum\value{mymac@matlab}<1%
	\textsuperscript{\textregistered}%
	\setcounter{mymac@matlab}{1}%
	\fi%
}

\theoremstyle{plain}% Theorem-like structures provided by amsthm.sty

\theoremstyle{remark}

\newcommand{\T}{\mathrm{T}}

\newcommand{\bSigma}{\boldsymbol{\Sigma}}

\newcommand{\bcG}{\boldsymbol{\cG}}
\newcommand{\bcZ}{\boldsymbol{\cZ}}
\newcommand{\bcX}{\boldsymbol{\cX}}

\newcommand{\vect}{\mathrm{vec}}
\newcommand{\Ll}{\mathrm{l}}
\newcommand{\p}{\mathrm{p}}
\newcommand{\ii}{\mathrm{i}}
\newcommand{\bcC}{\boldsymbol{\cC}}
\newcommand{\bcO}{\boldsymbol{\cO}}

\usepackage[normalem]{ulem}
\usepackage{xcolor}

\usepackage{graphicx}

\usepackage{amssymb}
\usepackage{amsthm}
\usepackage{cleveref}
\newtheorem{theorem}{Theorem}[section]
\newtheorem{lemma}{Lemma}[section]

\newtheorem{definition}{Definition}[section]

\theoremstyle{definition}

\newtheorem{remark}{Remark}
\Crefformat{subsection}{\color{blue!70!green!100!black}Subsection~#2#1#3}

\begin{document}
  
%%%%%%%%%%%%%%%%%%%%%%%%%%%%%%%%%%%%%%%%%%%%%%%%%%%%%%%%%%%%%%%%%%%%%%%%%%%%%%%%
% PAPER INFORMATION.                                                           %
%%%%%%%%%%%%%%%%%%%%%%%%%%%%%%%%%%%%%%%%%%%%%%%%%%%%%%%%%%%%%%%%%%%%%%%%%%%%%%%%

\title{Balanced Truncation of Descriptor Systems with a Quadratic Output}
  
\author[$\ast$, $\blacktriangle$]{Jennifer Przybilla}
\affil[$\ast$]{Max Planck Institute for Dynamics of Complex Technical Systems, \newline Sandtorstra{\ss}e 1, 39106 Magdeburg, Germany.}
\affil[$\ast\ast$]{Otto von Guericke University Magdeburg, Fakult\"at f\"ur Mathematik, \newline Universit\"atsplatz 2, 39106 Magdeburg, Germany.}
\affil[$\blacktriangle$]{\email{przybilla@mpi-magdeburg.mpg.de}, \orcid{0000-0002-8703-8735}}
\author[$\ast$, $\dagger$]{Igor Pontes Duff}
\affil[$\dagger$]{\email{pontes@mpi-magdeburg.mpg.de}, \orcid{0000-0001-6433-6142}}
  \author[$\ast$, $\vartriangle$]{\newline Pawan Goyal}
\affil[$\vartriangle$]{\email{goyalp@mpi-magdeburg.mpg.de}, \orcid{0000-0003-3072-7780}}
    
  \author[$\ast$, $\ast\ast$, $\circledast$]{Peter Benner}
\affil[$\circledast$]{\email{benner@mpi-magdeburg.mpg.de}, \orcid{0000-0003-3362-4103}}

\shorttitle{BT of DAE{\_}Q Systems}
\shortauthor{J. Przybilla, I. Pontes Duff, P. Goyal, P. Benner}
\shortdate{}
 
\keywords{model order reduction, balanced truncation, differential-algebraic systems, quadratic output systems, system Gramians, reduced-order models.}

\abstract{%
This work discusses model reduction for differential-algebraic systems with quadratic output equations. 
Under mild conditions, these systems can be transformed into a Weierstra{\ss} canonical form and, thus, be decoupled into differential equations and algebraic equations. The corresponding decoupled states are referred to as proper and improper states.
Due to the quadratic function of the state as an output, the proper and improper states are coupled in the output equation, which imposes a challenge from a model reduction viewpoint. 
Keeping the coupling in mind, our goal in this work is to find important subspaces of the proper and improper states and to reduce the system accordingly. 
To that end, we first propose the system's matrices, the so-called Gramians, to characterize the system's dominant subspaces.
We pay particular attention to the computation of the observability Gramians that take into account the nonlinear coupling between the proper and the improper states. 
We furthermore show that the proposed Gramians are related to certain kernel functions, which are used to identify important subspaces.
This allows us to propose a reduction algorithm to obtain reduced-order systems by removing the subspaces that are difficult to reach, as well as, difficult to observe.
Moreover, we quantify the error between the full-order and reduced-order models and demonstrate the proposed methodology using three numerical experiments.
}

\novelty{

	\begin{itemize}
		\item  We discuss a balanced truncation approach for linear differential-algebraic equations with a quadratic output. 
		\item For this, we propose new Gramians, characterizing the importance of the state from the input-output view-point and show their connections to corresponding kernel functions that are used to identify important subspaces.
		\item Moreover, we discuss an algorithm to construct reduced-order models using these Gramians, and characterize the error between the full-order  and reduced-order models due to the truncation. 
		\item The performance of the proposed methodology is demonstrated using three numerical examples.
	\end{itemize}

}

\maketitle

%%%%%%%%%%%%%%%%%%%%%%%%%%%%%%%%%%%%%%%%%%%%%%%%%%%%%%%%%%%%%%%%%%%%%%%%%%%%%%%%
% PAPER CONTENT.                                                               %
%%%%%%%%%%%%%%%%%%%%%%%%%%%%%%%%%%%%%%%%%%%%%%%%%%%%%%%%%%%%%%%%%%%%%%%%%%%%%%%%

% \linenumbers

%% main text explaining text.}

\section{Introduction}
In this paper, we discuss balanced truncation for a class of descriptor systems with a quadratic output function of the form
\begin{align}\label{eq:DAE_q}
\begin{split}
\bE\dot{\bx}(t) &= \bA \bx(t) + \bB \bu(t), \\
\by(t) &= \bx(t)^{\T}\bM \bx(t),
\end{split}
\end{align} 
where $\bE,~\bA\in\Rnn$, $\bB \in\Rnm$, and $\bM\in\Rnn$, where the matrix $\bE$ is singular, and $\bM$ is assumed to be symmetric, i.e., $\bM=\bM^{\T}$.
Additionally, we assume that the matrix pencil $s\bE-\bA$ is regular---that is, $\det(s\bE- \bA)$ is not the zero polynomial.
The input vector, the state vector, and the output are denoted by $\bu(t)\in\Rm, ~ \bx(t)\in\Rn$ and $\by(t)\in\R$, respectively.
In the following, we also assume that all the finite eigenvalues of the matrix pencil $s\bE-\bA$ lie in the negative half-plane, i.e., the system is asymptotically stable. 
Note, that these systems can be interpreted as a special class of Wiener models.

Differential algebraic systems (DAEs) arise when for example, electrical circuits, thermal and diffusion processes, or multibody systems are modeled by methods such as finite elements or finite volumes.
%, or semi-discretization of partial differential equations.  \todo{PG: I do not think constraints are imposed by finite elements? They often are from physics.}
These systems involve dynamic constraints that lead to algebraic equations, and therefore analysis tools must be developed for them. 
%\todo{PG: I am not sure if I understand correctly what is meant here in "Theser systems ... developed for them."}
%\red{\sout{Thereby, the differential and the algebraic part are considered separately in order to analyze the system behavior and properties.}}\todo{Igor: I would remove this part} 
%
%\todo{I think it is going to fast to the quadratic output terms. Maybe write a bit more about linear DAE system (such as "this system incorporate dynamical constriants and different simulations and analysis tools need to be develop to them"). And make the next line to be a paragraph. }
The system \eqref{eq:DAE_q} appears particularly while investigating the variance or deviation of the state variable from a certain reference point, which can be represented as a quadratic function of the state.

Models exhibiting complex dynamic behavior, or coming from PDEs discretization, are often high-fidelity models, i.e., the dimension of the state vector $n$ is large, which makes the engineering design process computationally infeasible. 
As a remedy, we seek to employ model reduction techniques that  allow us to construct a low-dimensional model which closely resembles the dynamic behaviors of the high-fidelity model. 
Our goal, in this paper, is to construct reduced-order models for the original models \eqref{eq:DAE_q} while preserving the original structure.
Precisely, we aim to determine the reduced-order models of the form
\begin{subequations}\label{eq:redDAE_q}
\begin{align}
\widehat{\bE}\dot{\widehat{\bx}}(t) &= \widehat{\bA}\widehat{\bx}(t) + \widehat{\bB }\bu(t),\label{eq:redDAE_q_a} \\
\widehat{\by}(t) &= \widehat{\bx}(t)^{\T}\widehat{\bM}\widehat{\bx}(t),\label{eq:redDAE_q_b}
\end{align}
\end{subequations}
where $\widehat{\bE}, ~\widehat{\bA}\in\Rrr$, $\widehat{\bB }\in\Rrm$ and $\widehat{\bM}\in\Rrr$, with $\widehat{\bM}=\widehat{\bM}^{\T}$ and $r\ll n$.
We obtain the reduced matrices in \eqref{eq:redDAE_q} by multiplying the matrices of system \eqref{eq:DAE_q} using two projection matrices, namely, $\bW_{\rr},~\bT_{\rr}\in\Rnr$, i.e., 
\begin{align*}
\widehat{\bE}=\bW_{\rr}^{\T}\bE \bT_{\rr}, \quad \widehat{\bA}=\bW_{\rr}^{\T}\bA \bT_{\rr}, \quad \widehat{\bB }=\bW_{\rr}^{\T}\bB , \quad \widehat{\bM}=\bT_{\rr}^{\T}\bM \bT_{\rr}.
\end{align*}
Furthermore, the reduced state and the approximated output are denoted by $\widehat{\bx}(t)\in\Rr$ and $\widehat{\by}(t)\in\mathbb{R}$, respectively.
The reduced-order model \eqref{eq:redDAE_q} shall approximate the input--to--output behavior of the full-order model \eqref{eq:DAE_q}, i.e., $\|\by-\widehat{\by}\| \leq \mathtt{tol}$, where $\mathtt{tol}$ is a user-defined tolerance,  for all admissible inputs $\bu$.

For ordinary differential equation (ODE) systems with a linear output equation, there exist several methods to construct reduced-order models, e.g.,  singular value-based approaches such as balanced truncation \cite{morMoo81,TomP87,morBenOCetal17} and Hankel norm approximations \cite{Glo84}. Moreover, moment matching methods \cite{morAntBG20,morLinBW07,morGugAB08} and Krylov subspace methods, e.g., the iterative rational Krylov algorithm (IRKA) \cite{morGugAB08, morGugSW13, morBenOCetal17, morFlaBG12} are frequently used.
A vast overview of these methods is given, e.g., in \cite{morBenGQetal21,morBenGQetal21a,morAntBG20,morBenOCetal17}.

All of the above-mentioned methods treat the case in which $\bE$ is nonsingular, and, therefore, are not directly applicable to the DAE case. 
There are several challenges that arise due to the algebraic equations.
Since the matrix $\bE$ is singular, the transfer function
$\bcG(s):= \bC(s\bE-\bA)^{-1}\bB $, defining the input-output mapping in the frequency domain, can have a non-zero polynomial part that needs to be preserved.
Therefore, a model reduction scheme for DAEs must preserve the polynomial part of its transfer function in the construction of reduced-order models. 
This issue is addressed in, e.g., \cite{morMehS05,morSty04,morGugSW13,morBenS17}.
Several existing methods that deal with the DAE case include interpolatory projection methods \cite{morGugSW13,morAhmB14,morAhmBG17,morAhmBGetal17} and balancing-based methods \cite{morSty04,morSty06,morHeiSS08,morMehS05,morBenQQ05}. 
In this work, we focus on a balancing-based method.
In balanced truncation (BT), one has to solve large generalized Lyapunov equations, also referred to as \emph{projected Lyapunov equations}, that act in the left and right deflating subspaces corresponding to the finite or infinite eigenvalues of the matrix pencil $s\bE - \bA$.
It requires to define the corresponding projection matrices that describe these deflating subspaces. 
However, such projection matrices are difficult to form explicitly.
Even if one manages, they can destroy the sparsity of the original matrices, thus, increasing the computational burden.
However, the structure of the DAE systems is often known and can be used to define, and implicitly apply the projection matrices in theory without the need of explicitly forming or multiplying by these projection matrices. 
For details, we refer to \cite{Sty08, morHeiSS08, morSaaV18, morBenSU16, morBenQQ05}.

However, BT is not directly applicable to the case of quadratic output functionals since the observability space is not of the same form as in the linear output case. Hence. the observability Gramian, defined in \cite{morMehS05}, is not usable.
In this work, we develop BT for DAEs with quadratic output equations. 
To that end, we derive new Gramians and corresponding kernel functions that allow us to characterize the controllability and observability behavior.
It is worth mentioning that in \cite{morBenGP21}, the authors derived Gramians corresponding to ODE systems (meaning $\bE = \bI$ in \eqref{eq:DAE_q}) with quadratic output equations. 
However, the methodology proposed in \cite{morBenGP21} cannot be directly applied to DAEs due to the singularity of the matrix $\bE$. 
%\todo[inline]{We could stress the difficulties of this work here and why the methods proposed in the literature are not straightforwardly applicable. } 
Therefore, there is a necessity to modify BT to incorporate the differential-algebraic structure. 
Precisely, we tailor the Gramians, corresponding to the proper and improper states of the system in \eqref{eq:DAE_q} that describe the controllability and observability spaces. Based on this, we proposed a balancing scheme to determine projection matrices $\bW_{\rr}$ and $\bT_{\rr}$, leading to the construction of reduced-order models.

The remainder of the paper is organized as follows. In \Cref{sec:Prem}, we briefly recap BT for differential-algebraic systems with linear output equations from \cite{morSty04}.
Moreover, we provide an overview of the Gramians for ODE systems with quadratic output equations, proposed in \cite{morBenGP21}.
%This theory is not directly applicable but provides the strategy for deriving the observability Gramians for system \eqref{eq:DAE_q}.
In \Cref{sec:Gramians}, we then derive the Gramians for the system \eqref{eq:DAE_q}. An algorithm to construct reduced-order models by truncating unimportant subspaces is discussed in \Cref{sec:Reduction}.
In \Cref{sec:ErrEst}, we derive an error estimator that bounds the output error between the original and reduced-order models.
Furthermore, in \Cref{sec:mulOutGram}, we extend the theory presented in \Cref{sec:Gramians,sec:Reduction,sec:ErrEst} for the systems with multiple inputs.
We demonstrate the efficiency of the method in \Cref{sec:NumRes} using three numerical examples.
We conclude the paper with a summary and future directions.

\section{Preliminaries}\label{sec:Prem}
In this section, we summarize previous works that form the basis for this paper. 
We begin by discussing the Weierstra\ss-canonical form for DAEs in \Cref{ssec:WCF} followed by the introduction to the BT method for DAE systems with linear output equation in \Cref{ssec:DAE_LO} and ODE systems with quadratic output equation in \Cref{ssec:LinSys_QuadOut}.

\subsection{Weierstra\ss-canonical form}\label{ssec:WCF}
We consider descriptor systems with a quadratic output function described in \eqref{eq:DAE_q}.
According to \cite{KunM06}, there exist matrices $\bW$ and $\bT$ that transform the differential-algebraic equation of the system \eqref{eq:DAE_q} into a Weierstra{\ss}-canonical form, i.e., 
\[
\bE = \bW\begin{bmatrix}
\bI_{n_f} & 0 \\
0 & \bN
\end{bmatrix}\bT, \quad \bA = \bW\begin{bmatrix}
\bJ & 0 \\
0 & \bI_{n_{\infty}}
\end{bmatrix}\bT, \quad \bB  = \bW\begin{bmatrix}
\bB _1 \\
\bB _2
\end{bmatrix}, \quad \bM = \bT^{\T}\begin{bmatrix}
\bM_{11} & \bM_{12} \\
\bM_{12}^{\T} & \bM_{22}
\end{bmatrix}\bT,
\]
with $n_f$ and $n_{\infty}$ being the respective numbers of the finite and infinite eigenvalues of the matrix pencil $s\bE-\bA$.
The matrix $\bJ\in\mathbb{R}^{n_f\times n_f}$ is in Jordan normal form, and $\bN\in\mathbb{R}^{n_{\infty}\times n_{\infty}}$ is nilpotent of nilpotency index $\nu$. Typically, the index $\nu$ is referred to as the index of the system \eqref{eq:DAE_q} as well. 
Moreover, we define the spectral projection matrices 
\begin{equation}\label{eq:Proj}
\bP_{\rr} = \bT^{-1}\begin{bmatrix}
\bI_{n_f} & 0\\
0 & 0
\end{bmatrix}\bT\quad\text{ and }\quad \bP_{\Ll} = \bW\begin{bmatrix}
\bI_{n_f} & 0\\
0 & 0
\end{bmatrix}\bW^{-1}
\end{equation}
onto the right and left deflating subspaces of the pencil $\lambda \bE-\bA$, corresponding to the finite eigenvalues.
By multiplying the system \eqref{eq:DAE_q} from the left by $\bW^{-1}$ and replacing $\bx(t) =:\bT^{-1}\begin{bmatrix}
\bx_1(t) \\
\bx_2(t)
\end{bmatrix}$, we obtain the following system:
\begin{subequations}\label{eq:DAE_q_WCF}
\begin{align}
\dot{\bx}_1(t) &= \bJ{\bx}_1(t)+\bB _1\bu(t),\label{eq:DAE_q_WCF_a} \\
\bN\dot{\bx}_2(t)&={\bx}_2(t) + \bB _2\bu(t),\label{eq:DAE_q_WCF_b} \\
\by(t) &= \begin{bmatrix}
{\bx}_1(t) \\
{\bx}_2(t)
\end{bmatrix}^{\T}\begin{bmatrix}
\bM_{11} & \bM_{12} \\
\bM_{12}^{\T} & \bM_{22}
\end{bmatrix}\begin{bmatrix}
{\bx}_1(t) \\
{\bx}_2(t)
\end{bmatrix}. \label{eq:DAE_q_WCF_c}
\end{align}
\end{subequations}
The system \eqref{eq:DAE_q_WCF} provides the decoupled proper and improper states $\bx_1(t)$ and $\bx_2(t)$.
Additionally, the solution trajectories of \eqref{eq:DAE_q_WCF_a} and \eqref{eq:DAE_q_WCF_b} are given as follows:
\begin{equation}\label{eq:solutionsWCF}
\bx_1(t) = \int_{0}^{t} e^{\bJ(t-\tau)}\bB _1\bu(\tau)\dd\tau, \qquad \bx_2(t) = \sum_{k=0}^{\nu -1} -\bN^k\bB _2 \bu^{(k)}(t),
\end{equation}
where $\bu^{(k)}(t)$ describes the $k$-th derivative of the function $\bu(\cdot)$ evaluated in the time variable $t$.
It is easy to see that initial conditions must satisfy $\bx_2(0) = \sum_{k=0}^{\nu -1} -\bN^k\bB _2 \bu^{(k)}(0)$. In this case the initial state is called \emph{consistent}.

Furthermore, we define 
\begin{align}\label{eq:F_J_F_N}
\bF_J(t):=\bT^{-1}\begin{bmatrix}
e^{\bJ t} & 0\\
0 & 0
\end{bmatrix}\bW^{-1} \qquad \text{and}\qquad \bF_N(k):=\bT^{-1}\begin{bmatrix}
0 & 0\\
0 & -\bN^k
\end{bmatrix}\bW^{-1}
\end{align}
and transform $\bx_1(t)$ and $\bx_2(t)$ into the original state space of system \eqref{eq:DAE_q} to obtain the proper and improper states 
\begin{equation}\label{eq:solutionsDAE}
\bx_{\p}(t) = \int_{0}^{t}\bF_J(t-\tau)\bB \bu(\tau)\dd\tau \;\;  \text{ and }\;\;   \bx_{\ii}(t) = \sum_{k=0}^{\nu -1}\bF_N(k)\bB \bu^{(k)}(t)
\end{equation}
with $\bx(t) = \bx_{\p}(t) + \bx_{\ii}(t)$.

In this paper, we aim to define Gramians that describe the controllability and the observability of the proper and improper states $\bx_{\p}(t)$ and $\bx_{\ii}(t)$.
Using these Gramians, we then derive tailored energy functional estimations that are used to identify irrelevant states. 
Note that the Weierstra{\ss}-canoncial form will only serve as a tool for analysis, but will not be computed in practice as its numerical determination is known to be difficult.

\subsection{BT for DAEs with linear output}\label{ssec:DAE_LO}
Here, we describe model reduction by BT for descriptor systems as introduced in \cite{morMehS05}.
We consider the continuous-time descriptor system with a linear output equation
\begin{align}\label{eq:DAE_l}
\begin{split}
\bE\dot{\bx}(t) &= \bA \bx(t) + \bB \bu(t),\qquad \bx(0) = 0, \\
\by(t) &= \bC \bx(t),
\end{split}
\end{align}
where the matrices $\bE, ~\bA, ~\bB$, and the vectors $\bx(t), ~\bu(t)$ are as in the system \eqref{eq:DAE_q}. 
The output equation provides the output, $\by(t)\in\Rq$, that results from the output matrix $\bC\in\Rqn$ and the state $\bx(t)$.
We assume that system \eqref{eq:DAE_l} is asymptotically stable, i.e., all finite eigenvalues of the matrix pencil $s\bE-\bA$ lie in the negative half-plane.

As described in the above subsection, the state $\bx(t)$ of system \eqref{eq:DAE_l} can be decomposed as
$
\bx(t) = \bx_{\p}(t) + \bx_{\ii}(t)$ with $\bx_{\p}(t)$ and $\bx_{\ii}(t)$ as defined in \eqref{eq:solutionsDAE}.
%where $\bx_{\p}(t)$ corresponds to the proper part of the system and $\bx_{\ii}(t)$ to the improper one.
We define the corresponding input--to--state mappings 
\[
\bcC_{\p}(t):=\bF_J(t)\bB \qquad\text{ and }\qquad\bcC_{\ii}(k):=\bF_N(k)\bB\]
 of the system \eqref{eq:DAE_l} that describe the controllability of the corresponding states.
With the help of these mappings, we can define the corresponding controllability Gramians of the system \eqref{eq:DAE_l}.
They are defined as $\bcP_{\p} := \int_{0}^{\infty} \bcC_{\p}(t)\bcC_{\p}(t)^{\T} \dd t$ and $\bcP_{\ii} := \sum_{k=0}^{\nu -1} \bcC_{\ii}(k)\bcC_{\ii}(k)^{\T}$ and result in the following Gramian expressions:
\begin{equation}\label{eq:Gramian_contr}
\bcP_{\p} := \int_{0}^{\infty} \bF_J(t)\bB \bB^{\T}\bF_J(t)^{\T}\dd t, \quad \bcP_{\ii} := \sum_{k=0}^{\nu -1} \bF_N(k)\bB \bB ^{\T}\bF_N(k)^{\T}.
\end{equation}
The matrices $\bcP_{\p}$ and $\bcP_{\ii}$ span the controllability space of the states $\bx_{\p}(t)$ and $\bx_{\ii}(t)$. 
Furthermore, inserting the definitions of $\bF_J(t)$ and $\bF_N(k)$ in \eqref{eq:Gramian_contr} indicates the Gramians to be of  the following form
\begin{equation}\label{eq:Contr_Gram}
\bcP_{\p} = \bT^{-1}\begin{bmatrix}
\bcP_1 & 0 \\
0 & 0
\end{bmatrix}\bT^{-\T}, \qquad \bcP_{\ii} = \bT^{-1}\begin{bmatrix}
0 & 0 \\
0 & \bcP_2
\end{bmatrix}\bT^{-\T}
\end{equation}
where $\bcP_1 = \int_{0}^{\infty} e^{\bJ t}\bB _1\bB _1^{\T}e^{\bJ^{\T}t} \dd t$ and $\bcP_2= \sum_{k=0}^{\nu -1} \bN^k\bB _2\bB _2^{\T}(\bN^k)^{\T}$ are the controllability Gramians corresponding to $\bx_1(t)$ and $\bx_2(t)$.
Using the controllability Gramians, we can characterize the states that are difficult to reach or even unreachable, which play an important role in the reduction of the system.

The Gramians $\bcP_{\p}$ and $\bcP_{\ii}$ satisfy the following time-continuous and time-discrete projected Lyapunov equation
\begin{subequations}\label{eq:Ly_P}
\begin{align}
\bE\bcP_{\p} \bA^{\T} + \bA\bcP_{\p} \bE^{\T} &= -\bP_{\Ll}\bB \bB ^{\T}\bP_{\Ll}^{\T}, \qquad\qquad\qquad \bcP_{\p} = \bP_{\rr}\bcP_{\p}\bP_{\rr}^{\T}, \label{eq:Ly_Pp}\\
\bA \bcP_{\ii} \bA^{\T} - \bE\bcP_{\ii} \bE^{\T} &= (I-\bP_{\Ll})\bB \bB ^{\T}(I-\bP_{\Ll})^{\T}, \qquad\;\, 0 = \bP_{\rr}\bcP_{\ii}\bP_{\rr}^{\T} \label{eq:Ly_Pi}
\end{align}
\end{subequations}
where the projection matrices $\bP_{\Ll}$ and $\bP_{\rr}$ are as defined in \eqref{eq:Proj}.

To describe the observability of the system \eqref{eq:DAE_l}, we consider the 
state--to--output mappings \[
\bcO_{\p}(t):=\bC\bF_J(t)\qquad\text{ and }\qquad\bcO_{\ii}(k):=\bC\bF_N(k).
\]
These mappings are used to describe the observability of certain states and are therefore used to define the proper and improper observability Gramians as $\bcQ_{\p} := \int_{0}^{\infty} \bcO_{\p}(t)^{\T}\bcO_{\p}(t) \dd t$ and $\bcQ_{\ii} := \sum_{k=0}^{\nu-1} \bcO_{\ii}(k)^{\T}\bcO_{\ii}(k)$ such that we obtain
\begin{align*}
\bcQ_{\p} := \int_{0}^{\infty} \bF_J(t)^{\T}\bC^{\T}\bC \bF_J(t) \dd \tau,   \qquad
\bcQ_{\ii} := \sum_{k=0}^{\nu-1} \bF_N(k)^{\T}\bC^{\T}\bC \bF_N(k).
\end{align*}
The goal of BT is to determine simultaneously the states which are both hard to reach and hard to observe. 
In general, the states corresponding to small singular values of the controllability Gramians do not coincide with the states corresponding to small singular values of the observability Gramians.
Therefore, we need to balance the system first, i.e., we transform the system to obtain a balanced one.

\begin{definition}
We call the system \eqref{eq:DAE_l} \emph{balanced} if the Gramians satisfy 
\[
\bcP_{\p} = \bcQ_{\p} = \begin{bmatrix}
\bSigma & 0\\
0 & 0
\end{bmatrix}, \qquad \bcP_{\ii} = \bcQ_{\ii} = \begin{bmatrix}
0 & 0\\
0 & \bTheta
\end{bmatrix},
\]
where $\bSigma=\diag{\sigma_1, \dots, \sigma_{n_f}}$, and $\bTheta=\diag{\theta_1, \dots, \theta_{n_\infty}}$.
\end{definition}
Since all Gramians are symmetric and positive semi-definite, there exist factorizations
\[
\bcP_{\p} = \bR_{\p}\bR_{\p}^{\T}, \quad \bcQ_{\p} = \bL_{\p}^{\T}\bL_{\p}, \quad \bcP_{\ii} = \bR_{\ii}\bR_{\ii}^{\T}, \quad \bcQ_{\ii} = \bL_{\ii}^{\T}\bL_{\ii}.
\]
Next, We compute the singular value decompositions 
\begin{align*}
\bL_{\p}\bE \bR_{\p} &= \bU_{\p}\bSigma \bV_{\p}^{\T} = \begin{bmatrix}
\bU_{\p,1} & \bU_{\p,2}
\end{bmatrix}\begin{bmatrix}
\bSigma_1 & \\
& \bSigma_2
\end{bmatrix}\begin{bmatrix}
\bV_{\p,1}^{\T} \\
\bV_{\p,2}^{\T}
\end{bmatrix},\\ \bL_{\ii}\bA \bR_{\ii} &= \bU_{\ii}\bTheta \bV_{\ii}^{\T} = \begin{bmatrix}
\bU_{\ii,1} & \bU_{\ii,2}
\end{bmatrix}\begin{bmatrix}
\bTheta_1 & \\
& 0
\end{bmatrix}\begin{bmatrix}
\bV_{\ii,1}^{\T} \\
\bV_{\ii,2}^{\T}
\end{bmatrix}
\end{align*}
where $\bSigma=\mathrm{diag}(\sigma_1, \dots, \sigma_n)$, $\sigma_1\geq \dots\geq \sigma_n$, includes the proper Hankel singular values of the system.
The proper states that are simultaneously difficult to reach and to observe correspond to the smallest Hankel singular values that are $\bSigma_2$.
We truncate the corresponding states that lie in the spaces spanned by $\bU_{p,2}$ and $\bV_{p,2}$ 
by building the projection matrices 
\[
\bW_{\rr} = [\bL_{\p}^{\T}\bU_{\p,1}\bSigma_1^{-\frac{1}{2}},~\bL_{\ii}^{\T}\bU_{\ii,1}\bTheta^{-\frac{1}{2}}], \qquad 
\bT_{\rr} = [\bR_{\p}^{\T}\bV_{\p,1}\bSigma_1^{-\frac{1}{2}},~\bR_{\ii}^{\T}\bV_{\ii,1}\bTheta^{-\frac{1}{2}}]. 
\]
Note that additionally improper states that correspond to zero singular values in $\bTheta$, i.e., the states that lie in the spaces spanned by $\bU_{\ii, 2}$ and $\bV_{\ii,2}$, are truncated. %\todo{Even uncontrollable improper states are not truncated?}
Multiplying the matrices of the full-order model in \eqref{eq:DAE_l} by $\bW_{\rr}$ and $\bT_{\rr}$ leads to a reduced-order model
\begin{align*}
\widehat{\bE}
\dot{\widehat{\bx}}(t) &= \widehat{\bA}
\widehat{\bx}(t)+
\widehat{\bB }\bu(t),\\
\widehat{\by}(t) &=
\widehat{\bC}
\widehat{\bx}(t),
\end{align*}
where it can be shown that
$
\widehat{\bE}=\begin{bmatrix}
\bI & 0 \\
0 & \widehat{\bE}_2
\end{bmatrix} $ and $\widehat{\bA}=\begin{bmatrix}\widehat{\bA}_1 & 0 \\
0 & \bI
\end{bmatrix}
$
with $\widehat{\bA}_1$  being nonsingular and $\widehat{\bE}_2$ being nilpotent.
Consequently, the reduced-order model is inherently decoupled into a proper and an improper reduced state.
The quality of the approximation can be estimated as
\begin{equation*}
\|\bcG-\widehat{\bcG} \|_{\cH_{\infty} }\leq 2(\sigma_{r+1}+\dots + \sigma_{n_f}),
\end{equation*}
where $\bcG(s) := \bC(s\bE-\bA)^{-1}\bB $ and $\widehat{\bcG}(s) := \widehat{\bC}(s\widehat{\bE}-\widehat{\bA})^{-1}\widehat{\bB }$ are the transfer functions of the original and reduced-order models, respectively, and $\sigma_{\ii}$ is the $i$-th largest singular value, that is the $i$-th diagonal element of $\bSigma$. 

We emphasize that the controllability behavior of the model in \eqref{eq:DAE_l} is the same as for the model in \eqref{eq:DAE_q}, as the input--to--state mapping is the same. However,
the observability behavior of \eqref{eq:DAE_l} is not the same as for \eqref{eq:DAE_q}, due to the quadratic form of the output equations instead of the linear one in \eqref{eq:DAE_l}.

\subsection{BT for linear dynamical systems with quadratic outputs}\label{ssec:LinSys_QuadOut}
In this subsection, we briefly summarize BT for linear systems with quadratic outputs---introduced in \cite{morBenGP21}---for $\bE = \bI$.
These systems are of the form
\begin{align}\label{eq:ODE_q}
\begin{split}
\dot{\bx}(t) &= \bA \bx(t) + \bB \bu(t),\qquad \bx(0) = 0, \\
\by(t) &= \bx(t)^{\mathrm{T}}\bM \bx(t),
\end{split}
\end{align}
where the matrices are defined as in \eqref{eq:DAE_q}, except that the matrix $\bE$ is replaced by the identity matrix. 
The goal is to find projection matrices $\bW_{\rr},~\bT_{\rr}\in\Rnr$ such that the reduced-order system
\begin{align}\label{eq:redODE_q}
\begin{split}
\dot{\widehat{\bx}}(t) &= \widehat{\bA}\widehat{\bx}(t) + \widehat{\bB }\bu(t),\qquad \widehat{\bx}(0) = 0, \\
\widehat{\by}(t) &= \widehat{\bx}(t)^{\T}\widehat{\bM}\widehat{\bx}(t)
\end{split}
\end{align}
with 
$
\widehat{\bA}:=\bW_{\rr}^{\T}\bA \bT_{\rr}, ~ \widehat{\bB }:=\bW_{\rr}^{\T}\bB , ~ \widehat{\bM}:=\bT_{\rr}^{\T}\bM \bT_{\rr}
$,
approximates the input--to--output behavior of the model in \eqref{eq:ODE_q} well.
To achieve this goal, using BT, a new pair of Gramians is discussed in \cite{morBenGP21}, tailored for the model in \eqref{eq:ODE_q}.

The state $\bx(t)$ of the model in \eqref{eq:ODE_q} is given by $\bx(t)  = \int_{0}^{t}e^{\bA(t-\tau)}\bB \bu(\tau)\dd \tau. $
We define the input--to--state mapping as $\bcC(t) := e^{\bA t}\bB $.
The controllability space of the model in \eqref{eq:ODE_q} is described by the controllability Gramian $\bcP$, that is 
\[
\bcP:= \int_{0}^{\infty} \bcC(t)\bcC(t)^{\mathrm{T}}\mathrm{d}t= \int_{0}^{\infty} e^{\bA t}\bB \bB ^{\mathrm{T}}e^{\bA^{\mathrm{T}}t}\mathrm{d}t.
\]
To describe the observability of the system, we consider the output equation
\begin{align*}
\by(t) &= \bx(t)^{\mathrm{T}}\bM \bx(t) = \int_{0}^{t}\int_{0}^{t}\bu(t_1)^{\mathrm{T}}\bB ^{\mathrm{T}}e^{\bA^{\mathrm{T}}t{_1}}\bM e^{\bA t_2}\bB \bu(t_2)\dd t_1\dd t_2 \\
&= \int_{0}^{t}\int_{0}^{t}\vecop{\bB ^{\mathrm{T}}e^{\bA^{\mathrm{T}}(t-\tau_1)}\bM e^{\bA(t-\tau_2)}\bB }(\bu(\tau_2)\otimes \bu(\tau_1))\dd \tau_1\dd \tau_2.
\end{align*}
We identify the input--to--state mapping $\bcC(t)=e^{\bA t}\bB $ that was defined above and define the state--to--output mapping 
\[
\bcO(t_1, t_2) := \bB ^{\mathrm{T}}e^{\bA^{\mathrm{T}}t_1}\bM e^{\bA t_2}.
\]
The corresponding observability Gramian $\bcQ$ is defined as follows:
\begin{align*}
\bcQ &=\int_{0}^{\infty} \int_{0}^{\infty}\bcO(t_1, t_2)^{\mathrm{T}}\bcO(t_1, t_2)\dd t_1\dd t_2
= \int_{0}^{\infty}  e^{\bA^{\mathrm{T}}t_2}\bM\int_{0}^{\infty}e^{\bA t_1}\bB \bB ^{\mathrm{T}}e^{\bA^{\mathrm{T}}t_1}\mathrm{d}t_1 \bM e^{\bA t_2}\mathrm{d}t_2\\
&= \int_{0}^{\infty}e^{\bA^{\mathrm{T}}t_2}\bM\bcP \bM e^{\bA t_2}\mathrm{d}t_2.
\end{align*}
Because of the positive definiteness of $\bcP$ and $\bcQ$, we can compute Cholesky factorizations $\bcP=\bR\bR^{\T}$ and $\bcQ=\bS\bS^{\T}$.
We make use of the energy functionals $E_{\mathrm{c}}(\bx_0)$ that is the minimal energy that is needed to reach a state $\bx_0$ and $E_{\mathrm{o}}(\bx_0)$ that is the output energy that is produced by a nonzero initial condition $\bx_0$. 
They satisfy 
\begin{align*}
E_{\mathrm{c}}(\bx_0)&:=\min_{\tiny\begin{matrix}
\bx(-\infty)=0\\\bx(0) = 0
\end{matrix}} \|\bu\|_{L_2}^2 = \bx_0^{\T}\bcP^{-1} \bx_0, \\ 
E_{\mathrm{o}}(\bx_0)&:=  \int_{0}^{\infty}\|\by(t)\|^2 \dd t  \leq \bx_0^{\T}\bcQ \bx_0\quad\text{for}\quad \bx_0 = \bR \bw,\; \|\bw\|\leq 1
\end{align*}
and can be used to characterize the hard to reach and hard to observe states, which correspond to small singular values of $\bcP$ and $\bcQ$, respectively. To truncate such states simultaneously,  we compute the singular value decomposition
\[
\bS^{\T} \bE \bR = \bU_{\p}\bSigma \bV_{\p}^{\T} = \begin{bmatrix}
\bU_1 & \bU_2
\end{bmatrix}\begin{bmatrix}
\bSigma_1 & \\
& \bSigma_2
\end{bmatrix}\begin{bmatrix}
\bV_1^{\T} \\
\bV_2^{\T}
\end{bmatrix}.
\]
The singular values in $\bSigma_1$ and the corresponding most observable and reachable states lying in the spaces spanned by $\bV_1$ and $\bU_1$ are used to derive the projection matrices 
\[
\bW_{\rr} = \bS^{\T}\bU_1\bSigma_1^{-\frac{1}{2}}, \qquad 
\bT_{\rr}  = \bR^{\T}\bV_1\bSigma_1^{-\frac{1}{2}}.
\]
We multiply the model in \eqref{eq:ODE_q} by $\bW_{\rr}$ and $\bT_{\rr}$ and obtain the reduced-order model in  \eqref{eq:redODE_q} that satisfies the error bound 
\begin{align*}
\|\by-\widehat{\by}\|_{\cL_{\infty}} \leq  \sqrt{\trace{\bB ^{\T}\bcQ \bB  - 2\bB ^{\T}\bcZ \widehat{\bB } + \widehat{\bB }^{\T}\widehat{\bcQ} \widehat{\bB }}}\|\bu \otimes \bu\|_{\cL_2}
\end{align*}
with 
\[
\bA^{\T}\bcZ + \bcZ\widehat{\bA} = -\bM\bcX\widehat{\bM}, \qquad \bA\bcX+\bcX\widehat{\bA}^{\T} = -\bB \widehat{\bB }^{\T}.
\]
The matrix $\widehat{\bcQ}$ is the observability Gramian of the reduced-order model in \eqref{eq:redODE_q}.

In the following sections, we extend the methodology for the DAE systems as described in \eqref{eq:DAE_q} and derive the corresponding proper and improper observability Gramians.

\section{Gramians for DAE systems with quadratic output}\label{sec:Gramians}
We aim to extend BT for  DAE systems with quadratic output equations described in \eqref{eq:DAE_q}. Therefore, we require tailored Gramians encoding controllability and observability subspaces for this class of systems.
Since the nonlinearities appear only on the output equation, the controllability Gramians from \eqref{eq:Gramian_contr}
can be used to characterize controllability.
However,  the extension of observability Gramians for DAE systems with quadratic output is not straightforward. 
Hence, in this section, we propose new Gramians that describe the observability of the proper and improper states based on an output decomposition. 
Later, we use them in our proposed BT method for DAE systems with quadratic output.
To derive observability Gramians, we decompose the output in the following way 
\begin{align*}
\by(t) &= \bx_{\p}(t)^{\T}\bM \bx_{\p}(t) + \bx_{\p}(t)^{\T}\bM \bx_{\ii}(t) + \bx_{\ii}(t)^{\T}\bM \bx_{\p}(t) + \bx_{\ii}(t)^{\T}\bM \bx_{\ii}(t)\\
&=:\by_{\p\p}(t)+\by_{\p\ii}(t)+\by_{\ii\p}(t)+\by_{\ii\ii}(t),
\end{align*}
where $\bx_{\p}$ and $ \bx_{\ii}$ are, respectively, the proper and improper states.
Note that the components $\bx_{\p}(t)^{\T}\bM \bx_{\ii}(t)$ and $\bx_{\ii}(t)^{\T}\bM \bx_{\p}(t)$ coincide. However, we will treat them independently for the purposes of the following derivations.
The idea in the following is to investigate the four components of the output separately. 
Since the output is a superposition of the four components, the Gramians that describe the output components sum up to Gramians that describe the overall observability of the system.

For a better understanding, we can rewrite $\by(t)$ by defining the state depending function $\bC(\bx(t)):=\bx(t)^{\T}\bM$.
Applying this representation to the decomposed output yields
\[
\by(t) = \bC(\bx_{\p}(t)) \bx_{\p}(t) + \bC(\bx_{\p}(t)) \bx_{\ii}(t) + \bC(\bx_{\ii}(t))\bx_{\p}(t) + \bC(\bx_{\ii}(t))\bx_{\ii}(t).
\]
%where $\bC(\bx_{\p}(t)) \bx_{\ii}(t) = \bC(\bx_{\p}(t))\bx_{\ii}(t)$.
We observe, that the observability of the state $\bx_{\p}(t)$ in the output $\by_{\ii\p}(t)=\bC(\bx_{\ii}(t))\bx_{\p}(t)$ also depends on the reachability of $\bx_{\ii}(t)$.
On the other hand, the observability of the improper state $\bx_{\ii}(t)$ corresponding to $\by_{\p\ii}(t)=\bC(\bx_{\p}(t))\bx_{\ii}(t)$ depends on the reachability of $\bx_{\p}(t)$. 
Hence, the outputs $\by_{\ii\p}(t)=\by_{\p\ii}(t)$ encode two different observability properties.
Analogously, the outputs $\by_{\p\p}(t)$ and $\by_{\ii\ii}(t)$ encode the observability of a proper state depending on the reachability of the same, and the observability of an improper state depending on the reachability of an improper one.

In this section, we define proper observability Gramians encoding the observability behavior of the proper state $\bx_{\p}(t)$ corresponding to $\bC(\bx_{\p}(t))$ and $\bC(\bx_{\ii}(t))$ and improper observability Gramians describing the observability of the improper states $\bx_{\ii}(t)$ corresponding to $\bC(\bx_{\p}(t))$ and $\bC(\bx_{\ii}(t))$.
Because of the dependencies on the reachability of $\bx_{\p}(t)$ and $\bx_{\ii}(t)$ encoded by $\bC(\bx_{\p}(t))$ and $\bC(\bx_{\ii}(t))$, we expect that the observability Gramians will depend on the controllability Gramians $\bcP_{\p}$ and $\bcP_{\ii}$.

\subsection{Proper observability Gramian}\label{ssec:propObsGram}
In this subsection, we investigate the two outputs $\by_{\p\p}(t)$ and $\by_{\ii\p}(t)$ and their observability properties. 
We aim to describe the observability of the right proper state depending on the second (left) state in the quadratic output equation, which is in the first case proper and in the second case improper.

\subparagraph*{Proper-proper output}
We start investigating the first component of the output  $\by_{\p\p}(t) = \bx_{\p}(t)^{\T}\bM \bx_{\p}(t)$ that includes two proper states. 
%Inserting the solution trajectories of the states %$\bx_{\p}(t)$ 
%leads to
%%
%\begin{align*}
%\by_{\p\p}(t) = \bC(\bx_{\p}(t))&= 
%\int_{0}^{t}\int_{0}^{t} \bu(\tau_1)^{\T} \bB ^{\T}\bF_J(t-\tau_1)^{\T}\bM \bF_J(t-\tau_2)\bB \bu(\tau_2)\dd\tau_2\dd\tau_1\\
%&= \int_{0}^{t}\int_{0}^{t} \vecop{\bB ^{\T}\bF_J(t-\tau_1)^{\T}\bM \bF_J(t-\tau_2)\bB }\left(\bu(\tau_2)\otimes \bu(\tau_1)\right)\dd\tau_2\dd\tau_1.
%\end{align*}
%We identify the controllability mapping $\bcC_{\p}(t) = \bF_J(t)\bB $ that was defined in \Cref{ssec:DAE_LO} and that was used to build the corresponding proper controllability Gramian in \eqref{eq:Gramian_contr}.
We define the state--to--output mapping \[
\bcO_{\p\p}(t_1, t_2) := \bB ^{\T}\bF_J(t_1)\bM \bF_J(t_2)
\]
that is used to describe the observability corresponding to $\by_{\p\p}(t)$.
Based on this mapping, we define in the following the proper-proper observability Gramian $\bcQ_{\p\p}$ as
\begin{align*}
\bcQ_{\p\p}&:= \int_{0}^{\infty}\int_{0}^{\infty}\bcO_{\p\p}(t_1, t_2)^{\T}\bcO_{\p\p}(t_1, t_2)\dd t_1 \dd t_2\\
&=\int_{0}^{\infty}\int_{0}^{\infty}\bF_J(t_2)^{\T}\bM \bF_J(t_1)\bB \bB ^{\T}\bF_J(t_1)^{\T}\bM \bF_J(t_2) \dd t_1 \dd t_2\\
&=\int_{0}^{\infty} \bF_J(t_2)^{\T}\bM\bcP_{\p} \bM \bF_J(t_2)  \dd t_2
\end{align*}
which results in the following definition.
\begin{definition}\label{def:Qpp}
Consider the asymptotically stable DAE system with a quadratic output equation from \eqref{eq:DAE_q} and the corresponding proper controllability Gramian $\bcP_{\p}$ as defined in \eqref{eq:Gramian_contr}.
The proper-proper observability Gramian $\bcQ_{\p\p}$ corresponding to the output $\by_{\p\p}$ %=\bx_{\p}(t)^{\T}\bM \bx_{\p}(t)$ 
is defined as
\begin{align*}
\bcQ_{\p\p}
=\int_{0}^{\infty} \bF_J(t)^{\T}\bM\bcP_{\p} \bM \bF_J(t)  \dd t,
\end{align*}
where $\bF_J(t)$ is defined as in \eqref{eq:F_J_F_N}.
\end{definition}
The above defined Gramian $\bcQ_{\p\p}$ can be transformed into the Weierstra{\ss}-canonical representation \eqref{eq:DAE_q_WCF}.
For that we insert the function $\bF_J(t)$ leading to
\[
\bcQ_{\p\p}:=  \bW^{-\T}\begin{bmatrix}
\bQ_{11} & 0 \\
0 & 0
\end{bmatrix}\bW^{-1},
\]
where 
\begin{equation}\label{eq:Q11}
\bQ_{11}:=\int_{0}^{\infty} e^{\bJ^{\T} t} \bM_{11}\bcP_1 \bM_{11} e^{\bJ t}  \dd t
\end{equation} 
is the proper-proper observability Gramian of the state $\bx_1(t)$ from system \eqref{eq:DAE_q_WCF}.
The Gramian $\bQ_{11}$ corresponding to the ODE system was analyzed in \cite{morBenGP21}. 

\begin{theorem}\label{theo:LyQ_pp}
Consider the asymptotically stable DAE system with a quadratic output equation from \eqref{eq:DAE_q} and the corresponding proper controllability Gramian $\bcP_{\p}$ as defined in \eqref{eq:Gramian_contr}.
The proper observability Gramian $\bcQ_{\p\p}$ solves the projected Lyapunov equation
\[
\bE^{\T} \bcQ_{\p\p} \bA + \bA^{\T} \bcQ_{\p\p} \bE = -\bP_{\rr}^{\T}\bM\bcP_{\p}\bM \bP_{\rr},\qquad \bcQ_{\p\p} = \bP_{\Ll}^{\T}\bcQ_{\p\p}\bP_{\Ll}.
\]
where the projection matrices $\bP_{\Ll}$ and $\bP_{\rr}$ are defined as in \eqref{eq:Proj}.
\end{theorem}
\begin{proof}
We first show that the Gramian $\bQ_{11}$ defined in \eqref{eq:Q11} solves the Lyapunov equation 
\begin{equation}\label{eq:Ly_Q_pp_hat}
\bJ^{\T}\bQ_{11} + \bQ_{11}\bJ = - \bM_{11}\bcP_1 \bM_{11}.
\end{equation}
Therefore, we insert $\bQ_{11}$ into \eqref{eq:Ly_Q_pp_hat} and obtain 
\begin{align*}
\int_{0}^{\infty} \left( \bJ^{\T}e^{\bJ^{\T} t} \bM_{11}\bcP_1 \bM_{11} e^{\bJ t} + e^{\bJ^{\T} t} \bM_{11}\bcP_1 \bM_{11} e^{\bJ t}\bJ\right) \dd t = \left[ e^{\bJ^{\T} t} \bM_{11}\bcP_1 \bM_{11} e^{\bJ t}\right]_0^{\infty}  = - \bM_{11}\bcP_1 \bM_{11}.
\end{align*}
To prove the statement of the theorem, we first observe that the projection condition is naturally satisfied since  $\bcQ_{\p\p}$ is by definition equal to $\bW^{-\T}\begin{bmatrix}
\bQ_{11} & 0 \\
0 & 0
\end{bmatrix}\bW^{-1}$.
To prove that $\bcQ_{\p\p}$ satisfies the remaining Lyapunov equation, we insert the Weierstra{\ss}-canonical form of $\bE$ and $\bA$ and the definition of $\bP_{\rr}$ into the equation to obtain 
\begin{align*}
\bT^{\T}\begin{bmatrix}
\bI & 0 \\
0 & \bN^{\T}
\end{bmatrix}\begin{bmatrix}
\bQ_{11} & 0 \\
0 & 0
\end{bmatrix}\begin{bmatrix}
\bJ & 0 \\
0 & \bI
\end{bmatrix}\bT 
+ \bT^{\T}\begin{bmatrix}
\bJ^{\T} & 0 \\
0 & \bI
\end{bmatrix}&\begin{bmatrix}
\bQ_{11} & 0 \\
0 & 0
\end{bmatrix} \begin{bmatrix}
\bI & 0 \\
0 & \bN
\end{bmatrix}\bT\\
&=\bT^{\T}\begin{bmatrix}
\bQ_{11}\bJ & 0 \\
0 & 0
\end{bmatrix}\bT 
+ \bT^{\T}\begin{bmatrix}
\bJ^{\T}\bQ_{11} & 0 \\
0 & 0
\end{bmatrix}\bT\\
&= -\bT^{\T}\begin{bmatrix}
\bM_{11}\bcP_1 \bM_{11} & 0\\
0 & 0
\end{bmatrix}\bT\\
&=-\bP_{\rr}^{\T}\bM\bcP_{\p}\bM \bP_{\rr}.
\end{align*}
such that \eqref{eq:Ly_Q_pp_hat} implies the statement since $\bT$ is a nonsingular matrix.
\end{proof}
\Cref{theo:LyQ_pp} states that we can calculate Gramians by solving certain projected Lyapunov equations. 
Methods to solve projected Lyapunov equations can be found, e.g., in \cite{Sty02,Sty08,StyS12}.

\subparagraph*{Improper-proper output}
Now we consider the third output component $\by_{\ii\p}(t) = \bx_{\ii}(t)^{\T}\bM \bx_{\p}(t)$.
%We insert the states $\bx_{\p}(t)$ and $\bx_{\ii}(t)$ and obtain
%\begin{align}\label{eq:kern_yip}
%\by_{\ii\p}(t) = \bC(\bx_{\ii}(t))\bx_{\p}(t) &= -\sum_{k=0}^{\nu-1}\int_{0}^{t} (\bu^{(k)}(t))^{\T} \bB ^{\T}\bF_N(k)^{\T}\bM \bF_J(t-\tau)\bB \bu(\tau)\dd \tau\nonumber\\
%&= -\sum_{k=0}^{\nu-1}\int_{0}^{t} \vecop{\bB ^{\T}\bF_N(k)^{\T}\bM \bF_J(t-\tau)\bB }\left(\bu(\tau)\otimes \bu^{(k)}(t)\right)\dd\tau.
%\end{align}
We define the state--to--output mapping $\bcO_{\ii\p}(t, k) := \bB ^{\T}\bF_N(k)^{\T}\bM \bF_J(t)$ that is used to describe the observability corresponding to $\by_{\ii\p}(t)$.
We note that we can also consider the transposed kernel to derive an improper Gramian, which is done later in Subsection~\ref{sec:Gramians_improper}.
However, we derive a Gramian that encodes the observability properties of the proper component of the output $\by_{\ii\p}(t)$.
We define the improper-proper observability Gramian as 
\begin{align*}
\bcQ_{\ii\p}&:= \sum_{k=0}^{\nu-1}\int_{0}^{\infty}\bcO_{\ii\p}(t, k)^{\T}\bcO_{\ii\p}(t, k)\dd t
=\sum_{k=0}^{\nu-1}\int_{0}^{\infty}\bF_J(t)^{\T}\bM \bF_N(k)\bB \bB ^{\T}\bF_N(k)^{\T}\bM \bF_J(t) \dd t\\
&=\int_{0}^{\infty} \bF_J(t)^{\T}\bM\bcP_{\ii}\bM \bF_J(t) \dd t.
\end{align*}
\begin{definition}\label{def:Qip}
Consider the asymptotically stable DAE system with a quadratic output equation from \eqref{eq:DAE_q} and the corresponding improper controllability Gramian $\bcP_{\ii}$ as defined in \eqref{eq:Gramian_contr}.
The improper-proper observability Gramian $\bcQ_{\ii\p}$ corresponding to the output $\by_{\ii\p}$ is defined as
\begin{align*}
\bcQ_{\ii\p}:=\int_{0}^{\infty} \bF_J(t)^{\T}\bM\bcP_{\ii}\bM \bF_J(t) \dd t
\end{align*}
where $\bF_J(t)$ is defined as in \eqref{eq:F_J_F_N}.
\end{definition}
The following theorem describes how the Gramian $\bcQ_{\ii\p}$ can be computed in practice.
\begin{theorem}\label{theo:LyQ_ip}
Consider the asymptotically stable DAE system with a quadratic output equation from \eqref{eq:DAE_q} and the corresponding improper controllability Gramian $\bcP_{\ii}$ as defined in \eqref{eq:Gramian_contr}.
The improper-proper observability Gramian $\bcQ_{\ii\p}$ solves the projected Lyapunov equation
\[
\bE^{\T} \bcQ_{\ii\p} \bA + \bA^{\T} \bcQ_{\ii\p} \bE = -\bP_{\rr}^{\T}\bM\bcP_{\ii}\bM \bP_{\rr},\quad \bcQ_{\ii\p} = \bP_{\Ll}^{\T}\bcQ_{\ii\p}\bP_{\Ll},
\]
where the projection matrices $\bP_{\Ll}$ and $\bP_{\rr}$ are defined as in \eqref{eq:Proj}.
\end{theorem}
\begin{proof}
The proof follows the same argumentation as for \Cref{theo:LyQ_pp}.
\end{proof}

\subparagraph*{Joined proper observability Gramian.}
We can combine the two proper output Gramians to obtain a Gramian that covers the observability of the proper states independent of the second state, that is the observability of the output $\by_{\p}(t) = \bx(t)^{\T}\bM \bx_{\p}(t)$ for an arbitrary state $\bx(t)$, generated by the model in \eqref{eq:DAE_q}.
Since the sum $\bcP_{\p} + \bcP_{\ii}$ spans the full controllability space of the state $\bx(t)$, the proper observability Gramian corresponding to both proper and improper left states is given by
\begin{align*}
\bcQ_{\p} &= \int_0^{\infty} \bF_J(t)^{\T}\bM(\bcP_{\p} + \bcP_{\ii})\bM \bF_J(t)\dd t\\
&= \int_0^{\infty} \bW^{-\T}\begin{bmatrix}
e^{\bJ^{\T}t} & 0 \\ 0 & 0
\end{bmatrix}\begin{bmatrix}
\bM_{11} & \bM_{12} \\ \bM_{12}^{\T} & \bM_{22}
\end{bmatrix}\begin{bmatrix}
\bcP_1 & 0 \\ 0 & \bcP_2
\end{bmatrix}\begin{bmatrix}
\bM_{11} & \bM_{12} \\ \bM_{12}^{\T} & \bM_{22}
\end{bmatrix}\begin{bmatrix}
e^{\bJ t} & 0 \\ 0 & 0
\end{bmatrix}\bW^{-1}\dd t\\
&=  \bW^{-\T}\begin{bmatrix}
\bQ_{11} + \bQ_{21} & 0 \\ 0 & 0
\end{bmatrix}\bW^{-1} = \bcQ_{\p\p} + \bcQ_{\ii\p}.
\end{align*}
We summarize this subsection with the following definition.
\begin{definition}\label{def:propGram}
Consider the asymptotically stable DAE system with a quadratic output equation from \eqref{eq:DAE_q} and the corresponding proper and improper controllability Gramians $\bcP_{\p}$ and $\bcP_{\ii}$ as defined in \eqref{eq:Gramian_contr}.
The proper observability Gramian $\bcQ_{\p}$ corresponding to the output $\by_{\p}=\by_{\p\p}(t) + \by_{\ii\p}(t)$ is defined as
\[
\bcQ_{\p} := \bcQ_{\p\p} + \bcQ_{\ii\p}
\]
with $\bcQ_{\p\p}$ and $\bcQ_{\ii\p}$ as in the \Cref{def:Qpp} and \ref{def:Qip}.   
\end{definition}
The summed proper Gramian $\bcQ_{\p}$ provides a criterion describing the states which are easiest and hardest to observe.
We will show later in \Cref{ssec:Reduction_input} that the hard to observe states are connected to the smallest non-zero singular values of the Gramian and thus serve as a truncation criterion to reduce the full-order model.

\subsection{Improper observability Gramians}\label{sec:Gramians_improper}
In this subsection, we investigate the observability behavior of the outputs $\by_{\p\ii}:= \bx_{\p}(t)^{\T} \bM \bx_{\ii}(t)$ and $\by_{\ii\ii}:= \bx_{\ii}(t)^{\T}\bM \bx_{\ii}(t)$.
Both outputs describe the observability of the improper state $\bx_{\ii}(t)$ corresponding to a proper and an improper state multiplied from the left.

\subparagraph*{Proper-improper output}

%The state $\by_{\p\ii}(t)$ is equal to 
%\begin{align*}
%\by_{\p\ii}(t) = \bC(\bx_{\p}(t))\bx_{\ii}(t)
%&= 
%\int_{0}^{t}\sum_{k=0}^{\nu-1} \bu(\tau)^{\T} \bB ^{\T}\bF_J(t-\tau)^{\T}\bM \bF_N(k)\bB \bu^{(k)}(t)\dd\tau\\
%&= \int_{0}^{t}\sum_{k=0}^{\nu-1} \vecop{\bB ^{\T}\bF_J(t-\tau)^{\T}\bM \bF_N(k)\bB }\left(\bu^{(k)}(t)\otimes \bu(\tau)\right)\dd\tau.
%\end{align*}
%We note again that $\by_{\p\ii}(t)$ is equal to $\by_{\ii\p}(t)$. 
We describe $\by_{\ii\p}(t)$ so that we derive an improper Gramian in this subsection.
This just means that we derive a Gramian that encodes the observability of the improper component of $\by_{\p\ii}(t)$.
We identify the improper controllability mapping $\bcC_{\ii}(k) = \bF_N(k)\bB $ and the remaining observability mapping $\bcO_{\p\ii}(t, k)=\bB ^{\T}\bF_J(t)^{\T}\bM \bF_N(k)$ which is used to define the proper-improper observability Gramian corresponding to the state $\bx_{\ii}(t)$ and the output $\by_{\p\ii}(t)$ as
\begin{align*}
\bcQ_{\p\ii} &= \int_0^{\infty}\sum_{k=0}^{\nu-1} \bcO_{\p\ii}(t, k)^{\T}\bcO_{\p\ii}(t, k)\dd t
= \int_0^{\infty}\sum_{k=0}^{\nu-1}\bF_N(k)^{\T} \bM \bF_J(t)\bB \bB ^{\T}\bF_J(t)^{\T}\bM \bF_N(k)\dd t\\
&= \sum_{k=0}^{\nu-1}\bF_N(k)^{\T} \bM\bcP_{\p}\bM \bF_N(k).
\end{align*}
This results in the following definition.
\begin{definition}\label{def:Qpi}
Consider the asymptotically stable DAE system with a quadratic output equation from \eqref{eq:DAE_q} and the corresponding proper controllability Gramian $\bcP_{\p}$ as defined in \eqref{eq:Gramian_contr}.
The proper-improper observability Gramian $\bcQ_{\p\ii}$ corresponding to the output $\by_{\p\ii}$ is defined as 
\[
\bcQ_{\p\ii} = \sum_{k=0}^{\nu-1}\bF_N(k)^{\T} \bM\bcP_{\p}\bM \bF_N(k),
\]
where $\bF_N(k)$ is defined as in \eqref{eq:F_J_F_N}.
\end{definition}
We insert the mapping $\bF_N(k)$ and obtain that the improper observability Gramian $\bcQ_{\p\ii}$ can be written as
\[
\bcQ_{\p\ii} = \bW^{-\T}\begin{bmatrix}
0 & 0 \\
0 & \bQ_{12}
\end{bmatrix}\bW^{-1}
\]
where 
\begin{equation}\label{eq:Q12}
\bQ_{12}:=\sum_{k=0}^{\nu-1}(-\bN^k)^{\T} \bM_{12}^{\T}\bcP_1\bM_{12} (-\bN^k)
\end{equation} is the proper-improper observability Gramian corresponding to the improper state $\bx_2(t)$ from the transformed model in \eqref{eq:DAE_q_WCF}.
This describes a relation between the Weiserstra{\ss}-canonical form \eqref{eq:DAE_q_WCF} and the full-order model in \eqref{eq:DAE_q} in this observability Gramain $\bcQ_{\p\ii}$.

\begin{theorem}\label{theo:LyQ_pi}
Consider the asymptotically stable DAE system with a quadratic output equation from \eqref{eq:DAE_q} and the corresponding proper controllability Gramian $\bcP_{\p}$ as defined in \eqref{eq:Gramian_contr}.
The proper-improper observability Gramian $\bcQ_{\p\ii}$ solves the projected generalized discrete-time Lyapunov equation
\[
\bA^{\T} \bcQ_{\p\ii} \bA - \bE^{\T} \bcQ_{\p\ii} \bE = (\bI-\bP_{\rr}^{\T})\bM\bcP_{\p}\bM(\bI-\bP_{\rr}),\quad \bP_{\Ll}^{\T}\bcQ_{\p\ii}\bP_{\Ll} = 0
\]
where the projection matrices $\bP_{\Ll}$ and $\bP_{\rr}$ are defined as in \eqref{eq:Proj}.
\end{theorem}
\begin{proof}
We first show that the Gramian $\bQ_{12}$ defined in \eqref{eq:Q12} solves the discrete-time Lyapunov equation 
\begin{equation}\label{eq:Ly_Q_pi_hat}
\bQ_{12} - \bN^{\T}\bQ_{12}\bN = \bM_{12}^{\T}\bcP_1\bM_{12}.
\end{equation}
That follows if we insert the definition of $\bQ_{12}$ into \eqref{eq:Ly_Q_pi_hat}.
This results in 
\begin{align*}
\sum_{k=0}^{\nu-1}(-\bN^k)^{\T} \bM_{12}^{\T}\bcP_1 \bM_{12} (-\bN^k) - \sum_{k=0}^{\nu-1}(-\bN^{k+1})^{\T} \bM_{12}^{\T}\bcP_1 \bM_{12} (-\bN^{k+1}) &= (-\bN^0)^{\T} \bM_{12}^{\T}\bcP_1\bM_{12} (-\bN^0) \\&= \bM_{12}^{\T}\bcP_1\bM_{12}
\end{align*}
since $\bN$ has the nilpotency index $\nu-1$, i.e., $\bN^{\nu}=0$.

To prove the projection condition, we derive
\begin{align*}
\bP_{\Ll}^{\T}\bcQ_{\p\p}\bP_{\Ll} = \bW^{-\T}\begin{bmatrix}
\bI & 0\\
0 & 0
\end{bmatrix}\bW^{\T}\bW^{-\T}\begin{bmatrix}
0 & 0 \\
0 & \bQ_{12}
\end{bmatrix}\bW^{-1}\bW\begin{bmatrix}
\bI & 0\\
0 & 0
\end{bmatrix}\bW^{-1}=0.
\end{align*}
To finalize the proof, we insert the Weierstra{\ss}-canonical form of $\bE$ and $\bA$ and the definition of $\bP_{\rr}$ into the remaining Lyapunov equation to obtain 
\begin{align*}
\bT^{\T}\begin{bmatrix}
\bJ^{\T} & 0 \\
0 & \bI
\end{bmatrix}\begin{bmatrix}
0 & 0 \\
0 & \bQ_{12}
\end{bmatrix}\begin{bmatrix}
\bJ & 0 \\
0 & \bI
\end{bmatrix}\bT - \bT^{\T}\begin{bmatrix}
\bI & 0 \\
0 & \bN^{\T}
\end{bmatrix}\begin{bmatrix}
0 & 0 \\
0 & \bQ_{12}
\end{bmatrix}&\begin{bmatrix}
\bI & 0 \\
0 & \bN
\end{bmatrix}\bT \\
&=\bT^{\T}\begin{bmatrix}
0 & 0 \\
0 & \bQ_{12}-\bN^{\T} \bQ_{12} \bN
\end{bmatrix}\bT\\
&=\bT^{\T}\begin{bmatrix}
0 & 0 \\
0 & \bM_{12}^{\T}\bcP_1\bM_{12}
\end{bmatrix}\bT\\
&= (\bI-\bP_{\rr}^{\T})\bM\bcP_{\p}\bM(\bI-\bP_{\rr})
\end{align*}
which proves the statement.
\end{proof}

\subparagraph*{Improper-improper output}
Now, we consider the fourth and last component of the output that describes the observability space of the improper state $\bx_{\ii}(t)$ if the second state is also improper.
%The output component $\by_{\ii\ii}(t)$ is equal to 
%\begin{align*}
%\by_{\ii\ii}(t) = \bC(\bx_{\ii}(t))\bx_{\ii}(t) &= 
%\sum_{k=0}^{\nu-1}\sum_{\ell=0}^{\nu-1} (\bu^{(k)}(t))^{\T} \bB ^{\T}\bF_N(k)^{\T} \bM \bF_N(\ell)\bB \bu^{(\ell)}(t)\\
%&= \sum_{k=0}^{\nu-1}\sum_{\ell=0}^{\nu-1} \vecop{\bB ^{\T}\bF_N(k)^{\T} \bM \bF_N(\ell)\bB }\left(\bu^{(\ell)}(t)\otimes \bu^{(k)}(t)\right).
%\end{align*}
We identify the state--to--output mapping $\bcO_{\ii\ii}(k, \ell)=\bB ^{\T}\bF_N(k)^{\T} \bM \bF_N(\ell)$.
Based on this mapping $\bcO_{\ii\ii}(k, \ell)$ we define the observability Gramian $\bcQ_{\ii\ii}$ as
\begin{align*}
\bcQ_{\ii\ii} &:= \sum_{k=0}^{\nu-1}\sum_{\ell=0}^{\nu-1} \bcO(k, \ell)^{\T}\bcO(k, \ell)
= \sum_{k=0}^{\nu-1}\sum_{\ell=0}^{\nu-1}\bF_N(\ell)^{\T} \bM \bF_N(k)\bB \bB ^{\T}\bF_N(k)^{\T}\bM \bF_N(\ell)\\
&= \sum_{\ell=0}^{\nu-1}\bF_N(\ell)^{\T} \bM\bcP_{\ii}\bM \bF_N(\ell)
\end{align*}
which results in the following definition.
\begin{definition}\label{eq:Qii}
Consider the DAE system with a quadratic output equation from \eqref{eq:DAE_q} and the corresponding improper controllability Gramian $\bcP_{\ii}$ as defined in \eqref{eq:Gramian_contr}.
The improper-improper observability Gramian $\bcQ_{\ii\ii}$ corresponding to the output $\by_{\ii\ii}$ is defined as 
\[
\bcQ_{\ii\ii} = \sum_{\ell=0}^{\nu-1}\bF_N(\ell)^{\T} \bM\bcP_{\ii}\bM \bF_N(\ell),
\]
where $\bF_N(\ell)$ is defined as in \eqref{eq:F_J_F_N}.
\end{definition}
\begin{theorem}\label{theo:LyQ_ii}
Consider the DAE system with a quadratic output equation from \eqref{eq:DAE_q} and the corresponding improper controllability Gramian $\bcP_{\ii}$ as defined in \eqref{eq:Gramian_contr}.
The improper-improper observability Gramian $\bcQ_{\ii\ii}$ solves the projected generalized discrete-time Lyapunov equation
\[
\bA^{\T} \bcQ_{\ii\ii} \bA - \bE^{\T} \bcQ_{\ii\ii} \bE = (\bI-\bP_{\rr}^{\T})\bM\bcP_{\ii}\bM(\bI-\bP_{\rr}),\quad \bP_{\Ll}^{\T}\bcQ_{\ii\ii}\bP_{\Ll} = 0
\]
where $\bP_{\Ll}$ and $\bP_{\rr}$ are defined as in \eqref{eq:Proj}.
\end{theorem}
\begin{proof}
The proof is similar to the proof of \Cref{theo:LyQ_pi}.
\end{proof}

\subparagraph*{Jointed improper observability Gramain}
We can combine the two improper output Gramians to obtain an improper Gramian that covers the observability of an improper state independent of the second state, that is, the observability of the output $\by_{\ii}(t) = \bx(t)^{\T}\bM \bx_{\ii}(t)$ for an arbitrary state $\bx(t)$ generated by system \eqref{eq:DAE_q}.
Since the sum $\bcP_{\p} + \bcP_{\ii}$ spans the full controllability space of the state $\bx(t)$, the improper observability Gramian corresponding to both proper and improper left states is given by
\begin{align*}
\bcQ_{\ii} &= \sum_{k=0}^{\nu -1} \bF_N(t)^{\T}\bM(\bcP_{\p} + \bcP_{\ii})\bM \bF_N(t)\\
&= \sum_{k=0}^{\nu -1} \bW^{-\T}\begin{bmatrix}
0 & 0 \\ 0 & -(\bN^k)^{\T}
\end{bmatrix}\begin{bmatrix}
\bM_{11} & \bM_{12} \\ \bM_{12}^{\T} & \bM_{22}
\end{bmatrix}\begin{bmatrix}
\bcP_1 & 0 \\ 0 & \bcP_2
\end{bmatrix}\begin{bmatrix}
\bM_{11} & \bM_{12} \\ \bM_{12}^{\T} & \bM_{22}
\end{bmatrix}\begin{bmatrix}
0 & 0 \\ 0 & -\bN^k
\end{bmatrix}\bW^{-1}\\
&= \bW^{-\T}\begin{bmatrix}
0 & 0 \\ 0 & \bQ_{12} + \bQ_{22}
\end{bmatrix}\bW^{-1} = \bcQ_{\p\ii} + \bcQ_{\ii\ii}.
\end{align*}
We summarize this subsection with the following definition.
\begin{definition}\label{def:impropGram}
Consider the DAE system with a quadratic output equation from \eqref{eq:DAE_q} and the corresponding proper and improper controllability Gramians $\bcP_{\p}$ and $\bcP_{\ii}$ as defined in \eqref{eq:Gramian_contr}.
The improper observability Gramian $\bcQ_{\ii}$ corresponding to the output $\by_{\ii}$ is defined as 
\[
\bcQ_{\ii} := \bcQ_{\p\ii} + \bcQ_{\ii\ii}
\]
where the Gramians $\bcQ_{\p\ii}$ and $\bcQ_{\ii\ii}$ are defined as in the \Cref{def:Qpi} and \ref{eq:Qii}.
\end{definition}

For the improper case, there are no energy functionals describing the connection of hard to observe states and the singular values of $\bcQ_{\ii}$.
However, improper states that lie in the kernel of $\bcQ_{\ii}$ on the subspace $\mathrm{ker}(\bP_{\Ll})$ are unobservable and, hence, can be removed from the dynamics without changing the input to output behavior. Therefore, the states corresponding to zero singular values of $\bcQ_{\ii}$ are removed in the reduction method presented in \Cref{ssec:BT}.

\section{Kernel functions and balanced truncation}\label{sec:Reduction}
In the previous section, we have proposed Gramians that describe the proper and improper controllability and observability spaces. 
In this section, the goal is to propose a BT method based on those Gramians. 
The proper and the improper states are considered separately.
We extend the methodology presented in \cite{morMehS05} to address systems with quadratic output equations.
For that, we first derive controllability and observability energies in \Cref{ssec:Reduction_input} that are used in \Cref{ssec:BT} to generate a reduced surrogate model that approximates the input--to--output behavior of the full-order model.

\subsection{Energy norms of kernel functions}\label{ssec:Reduction_input}

In standard BT theory, energy functionals are investigated to provide a criterion that states to truncate in the reduction step.
To evaluate the output energy, the $L_2$-norm of $\by(t)$ for an initial condition $\bx^*$ and a zero input $\bu(t)\equiv 0$ is considered.
However, since we consider consistent initial conditions, a vanishing input implies $\bx_{\ii}(t)\equiv 0$, and hence, $\by(t) = \by_{\p\p}(t)$.
The investigation of the output does not take into account the improper parts of the system and does not represent the complete system dynamics.

Hence, in this section, we investigate the dominant subspaces of the controllability mappings $\bcC_{\p}$ and $\bcC_{\ii}$ and of the observability mappings $\bcO_{\p\p}$, $\bcO_{\p\ii}$, $\bcO_{\ii\p}$ and $\bcO_{\ii\ii}$.
Afterward, in \Cref{ssec:BT}, we truncate the states living in the least important subspaces.

\subsubsection*{Controllability energy}

Consider the controllability mapping $\bcC_{\p}(t)$. We evaluate the energy norm of $\bcC_{\p}$ that is 
\begin{align*}
E(\bcC_{\p}) = \|\bcC_{\p}\|^2 = \trace{\int_{0}^{\infty}\bF_J(t)\bB\bB^{\T}\bF_J(t)^{\T}\dd t}
= \trace{\bcP_{\p}}
=\sigma_1+ \dots+ \sigma_{n_f},
\end{align*}
where $\sigma_1\geq \dots \geq \sigma_{n_f} \geq 0$ are eigenvalues of $\bcP_{\p}$.
Since $\bcP_{\p}$ is symmetric and positive semi-definite, there exists $\bV\in\Rnn$ with $\bV^{\T}\bV = \bI_n$ so that
$\bcP_{\p} = \bV\bSigma\bV^{\T}$ and $\bSigma = \diag{\sigma_1, \dots, \sigma_{n_f},0, \dots}$.
We see that the first $r$ columns of $\bV$ span the most dominant proper controllability subspace since they correspond to the largest eigenvalues of $\bcP_{\p}$ producing the largest energy values.
This observation justifies that in \Cref{ssec:BT}, the states corresponding to the smallest singular values of $\bcP_{\p}$ are truncated to reduce the model.

Similarly, we investigate the energy norm of the improper controllability mapping, that is 
\begin{align*}
E(\bcC_{\ii}) = \|\bcC_{\ii}\|^2 = \trace{\sum_{k = 0}^{\nu - 1}\bF_N(k)\bB\bB^{\T}\bF_N(k)^{\T}}
= \trace{\bcP_{\ii}}
=\theta_1+ \dots+ \theta_{n_{\infty}}.
\end{align*}
Again, because of the positive semi-definiteness of $\bcP_{\ii}$, we can decompose the Gramian into $\bcP_{\ii} = \bW\bTheta\bW^{\T}$ where J$\bTheta = \diag{\theta_1, \dots, \theta_{n_{\infty}}, 0\dots}$ includes the eigenvalues of $\bcP_{\ii}$ and $\bW^{\T}\bW=\bI$. 
As stated in \cite{morMehS05}, truncating states corresponding to small singular values of the improper Gramians already leads to inaccurate approximations since the non-zero singular values describe constraints to the model, and hence truncation of those could lead to physically meaningless results.
However, states corresponding to zero singular values of the improper Gramians correspond to unreachable improper states and can be truncated without changing the input--to--output behavior.

\subsubsection*{Observability energy}
To investigate the observability energies, we first gather the proper and improper observability mappings as 
\[
\bcO_{\p}(k, t_1, t_2):=\begin{bmatrix}
\bcO_{\p\p}(t_1, t_2)\\ \bcO_{\ii\p}(k, t_2)
\end{bmatrix}, \qquad 
\bcO_{\ii}(\ell, k, t):=\begin{bmatrix}
\bcO_{\p\ii}(\ell, t)\\ \bcO_{\ii\ii}(\ell, k)
\end{bmatrix}.
\]
We follow the same methodology as above and evaluate the energy norm of the proper observability mapping, which yields
\begin{align*}
E(\bcO_\p):=&\|\bcO_{\p}\|^2 = \trace{\sum_{k = 0}^{\nu-1}\int_{0}^{\infty}\int_{0}^{\infty}\bcO_{\p}(k, t_1, t_2)^{\T}\bcO_{\p}(k, t_1, t_2)\dd t_1 \dd t_2}\\
%=&\trace{\sum_{k = 0}^{\nu-1}\int_{0}^{\infty}\int_{0}^{\infty}\bcO_{\p\p}(t_1, t_2)^{\T}\bcO_{\p\p}(t_1, t_2) + \bcO_{\ii\p}(k, t_2)^{\T}\bcO_{\ii\p}(k, t_2)\dd t_1 \dd t_2}\\
=&\trace{\int_{0}^{\infty}\int_{0}^{\infty}\bcO_{\p\p}(t_1, t_2)^{\T}\bcO_{\p\p}(t_1, t_2) \dd t_1 \dd t_2
+\sum_{k = 0}^{\nu-1}\int_{0}^{\infty} \bcO_{\ii\p}(k, t_2)^{\T}\bcO_{\ii\p}(k, t_2)\dd t_2
}\\
=&\trace{\bcQ_{\p\p}} + \trace{\bcQ_{\ii\p}} = \trace{\bcQ_{\p}}
\end{align*}
with $\bcQ_{\p}$ as defined in \Cref{def:propGram}.
Accordingly, the improper energy norm is defined as 
\begin{align*}
E(\bcO_\ii):=&\|\bcO_{\ii}\|^2 = \trace{\sum_{\ell = 0}^{\nu-1}\sum_{k = 0}^{\nu-1}\int_{0}^{\infty}\bcO_{\ii}(\ell,k, t)^{\T}\bcO_{\ii}(\ell, k, t)\dd t}\\
%=&\trace{\sum_{\ell = 0}^{\nu-1}\sum_{k = 0}^{\nu-1}\int_{0}^{\infty}\bcO_{\p\ii}(k, t)^{\T}\bcO_{\p\ii}(k, t) + \bcO_{\ii\ii}(\ell,k)^{\T}\bcO_{\ii\ii}(\ell,k)\dd t}\\
=&\trace{\sum_{k = 0}^{\nu-1}\int_{0}^{\infty}\bcO_{\p\ii}(k, t)^{\T}\bcO_{\p\ii}(k, t)\dd t
+\sum_{\ell = 0}^{\nu-1}\sum_{k = 0}^{\nu-1}\bcO_{\ii\ii}(\ell,k)^{\T}\bcO_{\ii\ii}(\ell,k)
}\\
=& \trace{\bcQ_{\p\ii}} + \trace{\bcQ_{\ii\ii}} = \trace{\bcQ_{\ii}},
\end{align*}
where $\bcQ_{\ii}$ is the improper observability Gramian as defined in \Cref{def:impropGram}.
Consequently, in the BT method in the following section, we truncate subspaces corresponding to small eigenvalues of $\bcQ_{\p}$ and zero singular values of $\bcQ_{\ii}$.

\subsection{Balanced truncation}\label{ssec:BT}
In this subsection, we propose an extension of the BT method to the class of systems \eqref{eq:DAE_q}.
The Gramians and energies from the previous subsections provide a criterion that states to truncate.

We aim to truncate states corresponding to small eigenvalues of $\bcP_{\p}$ and $\bcQ_{\p}$.
However, in general, the states corresponding to small eigenvalues of $\bcP_{\p}$ and those corresponding to small eigenvalues of $\bcQ_{\p}$ do not coincide.
Hence, we need to balance the system before we truncate states.
\begin{definition}
The model in \eqref{eq:DAE_q} is called \emph{balanced} if the Gramians that are defined as in \eqref{eq:Gramian_contr} and in \Cref{def:propGram}, and \Cref{def:impropGram} satisfy
\[
\bcP_{\p} = \bcQ_{\p} = \begin{bmatrix}
\bSigma & 0 \\
0 & 0
\end{bmatrix}, \qquad \bcP_{\ii} = \bcQ_{\ii} = \begin{bmatrix}
0 & 0 \\
0 & \bTheta
\end{bmatrix},
\]
where $\bSigma = \diag{\sigma_1, \dots, \sigma_{n_f}}$ and $\bTheta = \diag{\theta_1, \dots, \theta_{n_{\infty}}}$.
\end{definition}
To balance the system \eqref{eq:DAE_q}, we need to define the following (low-rank) factors
\[
\bcP_{\p} = \bR_{\p}\bR_{\p}^{\T}, \quad \bcP_{\ii} = \bR_{\ii}\bR_{\ii}^{\T}, \quad \bcQ_{\p} = \bS_{\p}\bS_{\p}^{\T}, \quad \bcQ_{\ii} = \bS_{\ii}\bS_{\ii}^{\T}.
\]
%These factors exist since all Gramians are symmetric and positive semi-definite matrices.
Using these factors, we can compute the following singular value decompositions
\begin{align*}
\bL_{\p}\bE \bR_{\p} &= \bU_{\p}\bSigma \bV_{\p}^{\T} = \begin{bmatrix}
\bU_{\p,1} & \bU_{\p,2}
\end{bmatrix}\begin{bmatrix}
\bSigma_1 & \\
& \bSigma_2
\end{bmatrix}\begin{bmatrix}
\bV_{\p,1}^{\T} \\
\bV_{\p,2}^{\T}
\end{bmatrix},\\ \bL_{\ii}\bA \bR_{\ii} &= \bU_{\ii}\bTheta \bV_{\ii}^{\T} = \begin{bmatrix}
\bU_{\ii,1} & \bU_{\ii,2}
\end{bmatrix}\begin{bmatrix}
\bTheta_1 & \\
& 0
\end{bmatrix}\begin{bmatrix}
\bV_{\ii,1}^{\T} \\
\bV_{\ii,2}^{\T}
\end{bmatrix}.
\end{align*}
We can transform the system by the left and right projection matrix
\[
\bW_b = [\bL_{\p}^{\T}\bU_{\p}\bSigma^{-\frac{1}{2}},~\bL_{\ii}^{\T}\bU_{\ii}\bTheta^{-\frac{1}{2}}], \quad 
\bT_b = [\bR_{\p}^{\T}\bV_{\p}\bSigma^{-\frac{1}{2}},~\bR_{\ii}^{\T}\bV_{\ii}\bTheta^{-\frac{1}{2}}]
\]
to obtain a balanced system that has the form
\begin{align*}
\begin{bmatrix}
\bI & 0 \\
0 & \widetilde{\bE}_2
\end{bmatrix}\begin{bmatrix}
\dot{\widetilde{\bx}}_1(t)\\
\dot{\widetilde{\bx}}_2(t)
\end{bmatrix} &= \begin{bmatrix}\widetilde{\bA}_1 & 0 \\
0 & \bI
\end{bmatrix}\begin{bmatrix}
\widetilde{\bx}_1(t)\\
\widetilde{\bx}_2(t)
\end{bmatrix} + \begin{bmatrix}
\widetilde{\bB }_1\\
\widetilde{\bB }_2
\end{bmatrix}\bu(t),\\
\by(t) &= \begin{bmatrix}
\widetilde{\bx}_1(t)^{\T} &\widetilde{\bx}_2(t)^{\T}
\end{bmatrix}\begin{bmatrix}
\widetilde{\bM}_{11} & \widetilde{\bM}_{12}\\
\widetilde{\bM}_{12}^{\T} & \widetilde{\bM}_{22}
\end{bmatrix}\begin{bmatrix}
\widetilde{\bx}_1(t)\\
\widetilde{\bx}_2(t)
\end{bmatrix}.
\end{align*}
We see again a decomposition into a proper state $\widetilde{\bx}_1(t)$ and an improper state $\widetilde{\bx}_2(t)$, where the matrix $ \widetilde{\bE}_2$ is nilpotent.

We reduced the system by truncating the proper states $\widetilde{\bx}_1(t)$ that are most difficult to reach and to observe.
As we have seen in the previous subsection, these states correspond to the smallest singular values in $\bSigma$, i.e. $\bSigma_2$.
The projection matrices that balance the system and truncate these states simultaneously are
\[
\bW_{\rr} = [\bL_{\p}^{\T}\bU_{\p,1}\bSigma_1^{-\frac{1}{2}},~\bL_{\ii}^{\T}\bU_{\ii,1}\bTheta_1^{-\frac{1}{2}}], \quad 
\bT_{\rr} = [\bR_{\p}^{\T}\bV_{\p,1}\bSigma_1^{-\frac{1}{2}},~\bR_{\ii}^{\T}\bV_{\ii,1}\bTheta_1^{-\frac{1}{2}}].
\]
We multiply the full-order model in \eqref{eq:DAE_q} by the projection matrices $\bW_{\rr}$ and $\bT_{\rr}$ to obtain a reduced-order model \eqref{eq:redDAE_q} that is of the form
\begin{align}\label{eq:redSyst}
\begin{split}
\begin{bmatrix}
\bI & 0 \\
0 & \widetilde{\bE}_2
\end{bmatrix}\begin{bmatrix}
\dot{\widehat{\bx}}_1(t)\\
\dot{\widetilde{\bx}}_2(t)
\end{bmatrix} &= \begin{bmatrix}\widehat{\bA}_1 & 0 \\
0 & \bI
\end{bmatrix}\begin{bmatrix}
\widehat{\bx}_1(t)\\
\widetilde{\bx}_2(t)
\end{bmatrix} + \begin{bmatrix}
\widehat{\bB }_1\\
\widetilde{\bB }_2
\end{bmatrix}\bu(t),\\
\widehat{\by}(t) &= \begin{bmatrix}
\widehat{\bx}_1(t)^{\T} &\widetilde{\bx}_2(t)^{\T}
\end{bmatrix}\begin{bmatrix}
\widehat{\bM}_{11} & \widehat{\bM}_{12}\\
\widehat{\bM}_{12}^{\T} & \widehat{\bM}_{22}
\end{bmatrix}\begin{bmatrix}
\widehat{\bx}_1(t)\\
\widetilde{\bx}_2(t)
\end{bmatrix}
\end{split}
\end{align}
where $r$ is the dimension of the reduced proper space.
The reduction by BT is summarized in \Cref{algo:BT_DAEQ}. 
This algorithm follows the same line as the BT method for differential-algebraic systems with linear output presented in \cite{morMehS05}, besides the fact that the observability Gramians are different in our algorithm.

\begin{algorithm}[tb]
	\caption{BT method for for DAE systems with quadratic output.}
	\label{algo:BT_DAEQ}
	\begin{algorithmic}[1]
		\Require{The full-order model \eqref{eq:DAE_q} and the order $r$.}
		\Ensure{The reduced-order model \eqref{eq:redSyst}.}
		\State{Compute the proper and improper controllability Gramians $\bcP_{\p}$ and $\bcP_{\ii}$ by solving the Lyapunov equations in \eqref{eq:Ly_Pp}, \eqref{eq:Ly_Pi}.}
		\State{Compute the proper and improper observability Gramians $\bcQ_{\p}$ and $\bcQ_{\ii}$ by solving the Lyapunov equations from \Cref{theo:LyQ_pp,theo:LyQ_ip,theo:LyQ_pi,theo:LyQ_ii}.}
		\State Perform the singular values decompositions \[
		\bL_{\p}\bE \bR_{\p} = \begin{bmatrix}
		\bU_{\p,1} & \bU_{\p,2}
		\end{bmatrix}\begin{bmatrix}
		\bSigma_1 & \\
		& \bSigma_2
		\end{bmatrix}\begin{bmatrix}
		\bV_{\p,1}^{\T} \\
		\bV_{\p,2}^{\T}
		\end{bmatrix},\quad \bL_{\ii}\bA \bR_{\ii} = \begin{bmatrix}
		\bU_{\ii,1} & \bU_{\ii,2}
		\end{bmatrix}\begin{bmatrix}
		\bTheta_1 & \\
		& 0
		\end{bmatrix}\begin{bmatrix}
		\bV_{\ii,1}^{\T} \\
		\bV_{\ii,2}^{\T}
		\end{bmatrix}.
		\]
		\State Construct the projection matrices \[ \bW_{\rr} = [\bL_{\p}^{\T}\bU_{\p,1}\bSigma_1^{-\frac{1}{2}},~\bL_{\ii}^{\T}\bU_{\ii,1}\bTheta_1^{-\frac{1}{2}}], \quad 
		\bT_{\rr} = [\bR_{\p}^{\T}\bV_{\p,1}\bSigma_1^{-\frac{1}{2}},~\bR_{\ii}^{\T}\bV_{\ii,1}\bTheta_1^{-\frac{1}{2}}].
		\]
		\State Construct reduced matrices
		\begin{align*}
\widehat{\bE}:=\bW_{\rr}^{\T}\bE \bT_{\rr}, \quad \widehat{\bA}:=\bW_{\rr}^{\T}\bA \bT_{\rr}, \quad \widehat{\bB }:=\bW_{\rr}^{\T}\bB , \quad \widehat{\bM}:=\bT_{\rr}^{\T}\bM \bT_{\rr}.
		\end{align*}
	\end{algorithmic}
\end{algorithm}
\begin{remark}
The BT method presented above decouples the proper and improper states as described in \eqref{eq:redSyst} where the proper states are reduced while for the improper states, only a minimal realization is found.
That means that improper states corresponding to zero singular values of the improper Gramians are truncated since they are not reachable or not observable and do not change the input--to--output behavior of the system.
\end{remark}

\section{Error Estimation}\label{sec:ErrEst}
We aim to estimate the error between the output $\by$ and the reduced output $\widehat{\by}$ we obtain when we evaluate the reduced-order model \eqref{eq:redSyst}.
We estimate 
\[
\|\by-\widehat{\by}\|_{L_{\infty}}\leq \|\by_{\p\p}-\widehat{\by}_{\p\p}\|_{L_{\infty}} + \|\by_{\p\ii}-\widehat{\by}_{\p\ii}\|_{L_{\infty}} + \|\by_{\ii\p}-\widehat{\by}_{\ii\p}\|_{L_{\infty}} + \|\by_{\ii\ii}-\widehat{\by}_{\ii\ii}\|_{L_{\infty}}
\]
and consider the four summands separately.
Since we do not truncate the improper states the summand $\|\by_{\ii\ii}-\widehat{\by}_{\ii\ii}\|_{L_{\infty}}$ is equal to zero. The remaining summands are investigated in the following.

\subsection{The proper-proper output error}
In this section, we aim to analyze the error between the proper-proper output $\by_{\p\p}(t)$ and its approximation $\widehat{\by}_{\p\p}(t)$.
To do so, we define 
\begin{equation}\label{eq:hpp}
\bh_{\p\p}(t_1, t_2) := \vecop{\bB ^{\T}\bF_J(t_1)^{\T}\bM \bF_J(t_2)\bB }\quad\text{ and }\quad\widehat{\bh}_{\p\p}(t_1, t_2):=\vecop{\widehat{\bB }^{\T}\widehat{\bF}_J(t_1)^{\T}\widehat{\bM}\widehat{\bF}_J(t_2)\widehat{\bB}},
\end{equation}
where $\widehat{\bF}_J(t):=\begin{bmatrix}
e^{\widehat{\bA}_1t} & 0 \\
0 & 0
\end{bmatrix}$, so that the outputs can be represented as 
\begin{align*}
\by_{\p\p}(t) &=  \bx_{\p}(t)^{\T}\bM \bx_{\p}(t) %= \int_{0}^{t}\int_{0}^{t} \vecop{\bB ^{\T}\bF_J(t-t_1)^{\T}\bM \bF_J(t-t_2)\bB }
=\int_{0}^{t}\int_{0}^{t}\bh_{\p\p}(t_1, t_2)
(\bu(t_2)\otimes \bu(t_1))\dd t_1 \dd t_2,\\
\widehat{\by}_{\p\p}(t) &=  \widehat{\bx}_{\p}(t)^{\T}\widehat{\bM} \widehat{\bx}_{\p}(t) %$= \int_{0}^{t}\int_{0}^{t} \vecop{\widehat{\bB} ^{\T}\widehat{\bF}_J(t-t_1)^{\T}\widehat{\bM} \widehat{\bF}_J(t-t_2)\widehat{\bB} }\cdot(\bu(t_2)\otimes \bu(t_1))\dd t_1 \dd t_2
=\int_{0}^{t}\int_{0}^{t}\widehat{\bh}_{\p\p}(t_1, t_2)(\bu(t_2)\otimes \bu(t_1))\dd t_1 \dd t_2.
\end{align*}
Using these representations of $\by_{\p\p}$ and $\widehat{\by}_{\p\p}$ the following lemma provides an upper bound of the $L_{\infty}$-error in the proper proper output.
\begin{lemma}\label{lemma:Lintnorm_pp}
We consider the asymptotically stable DAE system with a quadratic output equation from \eqref{eq:DAE_q}, the reduced-order model in \eqref{eq:redSyst}, and $\bh_{\p\p}(t_1, t_2)$, $\widehat{\bh}_{\p\p}(t_1, t_2)$ as defined in \eqref{eq:hpp}.
Then, the following inequality holds
\begin{align*}
\|\by_{\p\p}-\widehat{\by}_{\p\p}\|_{L_{\infty}}\leq  \left(\int_{0}^{\infty} \int_{0}^{\infty}\|\bh_{\p\p}(t_1, t_2)-\widehat{\bh}_{\p\p}(t_1, t_2)\|_2^2\dd t_1 \dd t_2\right)^{\frac{1}{2}}\|\bu\otimes \bu\|_{L_2}.
\end{align*}
\end{lemma}
\begin{proof}
We consider the output error at time $t\geq 0$ that is
\begin{align*}
\big|\by_{\p\p}(t)-\widehat{\by}_{\p\p}(t)\big| &= \bigg|\int_{0}^{t} \int_{0}^{t}  \bx_{\p}(t_1)^{\T}\bM \bx_{\p}(t_2)
-\widehat{\bx}_{\p}(t_1)^{\T}\widehat{\bM}\widehat{\bx}_{\p}(t_2)\dd t_1 \dd t_2\bigg|\\
&= \bigg|\int_{0}^{t} \int_{0}^{t} \left(\bh_{\p\p}(t-t_1, t-t_2)-\widehat{\bh}_{\p\p}(t-t_1, t-t_2)\right)(\bu(t_2)\otimes \bu(t_1))\dd t_1 \dd t_2\bigg|.
\end{align*}
Applying the Cauchy-Schwarz inequality multiple times yields
\begin{align*}
|\by_{\p\p}(t)-\widehat{\by}_{\p\p}(t)|
&\leq \int_{0}^{t} \int_{0}^{t} \|\left(\bh_{\p\p}(t-t_1, t-t_2)-\widehat{\bh}_{\p\p}(t-t_1, t-t_2)\right)(\bu(t_2)\otimes \bu(t_1))\| \dd t_1 \dd t_2\\
&\leq \int_{0}^{t} \int_{0}^{t} \|\bh_{\p\p}(t_1, t_2)-\widehat{\bh}_{\p\p}(t_1, t_2)\|_2\|(\bu(t_2)\otimes \bu(t_1))\|_2 \dd t_1 \dd t_2\\
&\leq \left(\int_{0}^{t} \int_{0}^{t}\|\bh_{\p\p}(t_1, t_2)-\widehat{\bh}_{\p\p}(t_1, t_2)\|_2^2\dd t_1 \dd t_2\right)^{\frac{1}{2}}\left(\int_{0}^{t} \int_{0}^{t}\|(\bu(t_2)\otimes \bu(t_1))\|_2^2\dd t_1 \dd t_2\right)^{\frac{1}{2}}.
\end{align*}
Hence, we can bound the $L_{\infty}$-norm of the output error as 
\begin{align*}
\|\by_{\p\p}-&\widehat{\by}_{\p\p}\|_{L_{\infty}}\\
&\leq \left(\int_{0}^{\infty} \int_{0}^{\infty}\|\bh_{\p\p}(t_1, t_2)-\widehat{\bh}_{\p\p}(t_1, t_2)\|_2^2\dd t_1 \dd t_2\right)^{\frac{1}{2}}\left(\int_{0}^{\infty} \int_{0}^{\infty}\|(\bu(t_2)\otimes \bu(t_1))\|_2^2\dd t_1 \dd t_2\right)^{\frac{1}{2}}\\
&= \left(\int_{0}^{\infty} \int_{0}^{\infty}\|\bh_{\p\p}(t_1, t_2)-\widehat{\bh}_{\p\p}(t_1, t_2)\|_2^2\dd t_1 \dd t_2\right)^{\frac{1}{2}}\|\bu\otimes \bu\|_{L_2}.
\end{align*}
\end{proof}

%with 
%\begin{multline*}
%	\int_{0}^{\infty} \int_{0}^{\infty}|\bh_{\p\p}(t_1, t_2)-\widehat{\bh}_{\p\p}(t_1, t_2)\|_2^2\dd t_1 \dd t_2\\
%	= \int_{0}^{\infty} \int_{0}^{\infty}\|\bh_{\p\p}(t_1, t_2)\|_2^2  -2\langle \bh_{\p\p}(t_1, t_2),\widehat{\bh}_{\p\p}(t_1, t_2)\rangle  + \|\widehat{\bh}_{\p\p}(t_1, t_2)\|_2^2\dd t_1 \dd t_2.
%\end{multline*}

\begin{lemma}\label{lemma:traceLintnorm_pp}
We consider the asymptotically stable DAE system with a quadratic output equation from \eqref{eq:DAE_q}, the reduced-order model in \eqref{eq:redSyst}, the corresponding proper controllability Gramian $\bcP_{\p}$ as defined in \eqref{eq:Gramian_contr}, and the reduced proper controllability Gramian 
\[
\widehat{\bcP}_{\p}:= \int_0^{\infty}\begin{bmatrix}
e^{\widehat{\bA}_1t}\widehat{\bB }_1\widehat{\bB }_1^{\T}e^{\widehat{\bA}_1^{\T}t} & 0 \\ 0 & 0
\end{bmatrix}\dd t.
\] The functionals $\bh_{\p\p}(t_1, t_2)$ and $\widehat{\bh}_{\p\p}(t_1, t_2)$ are as defined in \eqref{eq:hpp}.
Then, the following equalities hold
\begin{align}
\int_{0}^{\infty} \int_{0}^{\infty}\|\bh_{\p\p}(t_1, t_2)\|_2^2\dd t_1 \dd t_2 &= \trace{\bcP_{\p}\bM\bcP_{\p}\bM},\label{eq:PPeq1}\\
\int_{0}^{\infty} \int_{0}^{\infty}\|\widehat{\bh}_{\p\p}(t_1, t_2)\|_2^2\dd t_1 \dd t_2 &= \trace{\widehat{\bcP}_{\p}\widehat{\bM}\widehat{\bcP}_{\p}\widehat{\bM}},\label{eq:PPeq2}\\
\int_{0}^{\infty} \int_{0}^{\infty}\langle \bh_{\p\p}(t_1, t_2),\widehat{\bh}_{\p\p}(t_1, t_2)\rangle\dd t_1 \dd t_2 &= \trace{\widetilde{\bcP}_{\p}^{\T}\bM\widetilde{\bP}_{\p}\widehat{\bM}}\label{eq:PPeq3}
\end{align}
where $\widetilde{\bcP}_{\p}:= \int_{0}^{\infty}\bF_J(t)\bB \widehat{\bB}^{\T}\widehat{\bF}_J(t)^{\T}\dd t$ satisfies the projected Sylvester equation 
\begin{align}\label{eq:PPSylv}
\bA\widetilde{\bcP}_{\p} \widehat{\bE}^{\T} + \bE\widetilde{\bcP}_{\p}\widehat{\bA}^{\T} = - \bP_{\Ll} \bB \widehat{\bB}^{\T}\widehat{\bP}_{\Ll}^{\T}, \qquad \widetilde{\bcP}_{\p} = \bP_{\rr} \widetilde{\bcP}_{\p} \widehat{\bP}_{\rr}^{\T}
\end{align}
with $\widehat{\bP}_{\Ll} = \widehat{\bP}_{\rr} = \begin{bmatrix}
\bI_{\rr} & 0 \\ 0 & 0
\end{bmatrix}$, $\bP_{\Ll}$ and $\bP_{\rr}$ as defined in \eqref{eq:Proj}, and $\langle\cdot,\cdot\rangle$ denotes the Euclidean inner product with $\langle v_1,v_2\rangle = v_1^{\T}v_2$ for $v_1,v_2\in \R^{n}$.
\end{lemma}
\begin{proof}
We make use of the property $\|\vect(X)\|_2^2 = \|X\|_{F}^2$ and the Kronecker product properties to obtain
\begin{align*}
\int_{0}^{\infty} \int_{0}^{\infty}\|\bh_{\p\p}(t_1, t_2)\|_2^2\dd t_1 \dd t_2 
&= \int_{0}^{\infty} \int_{0}^{\infty}\trace{\bB ^{\T}\bF_J(t_2)^{\T}\bM \bF_J(t_1)\bB \bB ^{\T}\bF_J(t_1)^{\T}\bM \bF_J(t_2)\bB }\dd t_1 \dd t_2 \\
&= \int_{0}^{\infty}\trace{\bB ^{\T}\bF_J(t_2)^{\T}\bM\bcP_{\p}\bM \bF_J(t_2)\bB }\dd t_2 \\
&= \int_{0}^{\infty}\trace{\bF_J(t_2)\bB \bB ^{\T}\bF_J(t_2)^{\T}\bM\bcP_{\p}\bM}\dd t_2 \\
&= \trace{\bcP_{\p}\bM\bcP_{\p}\bM},
\end{align*}
which proves \eqref{eq:PPeq1} while \eqref{eq:PPeq2} is proven analogously.
To show the last equation given in \eqref{eq:PPeq3} we make use of the property $\langle\vect(X),\vect(Y)\rangle = \trace{X^{\T}Y}$ and obtain 
\begin{align*}
\int_{0}^{\infty} \int_{0}^{\infty}\langle \bh_{\p\p}(t_1, t_2),\widehat{\bh}_{\p\p}&(t_1, t_2)\rangle\dd t_1 \dd t_2\\
&=\int_{0}^{\infty} \int_{0}^{\infty}
\trace{\bB^{\T}\bF_J(t_2)^{\T}\bM \bF_J(t_1)\bB \widehat{\bB }^{\T}\widehat{\bF}_J(t_1)^{\T}\widehat{\bM}\widehat{\bF}_J(t_2)\widehat{\bB }}\dd t_1 \dd t_2\\
&=\int_{0}^{\infty} \int_{0}^{\infty}
\trace{\widehat{\bF}_J(t_2)\widehat{\bB }\bB^{\T}\bF_J(t_2)^{\T}\bM \bF_J(t_1)\bB \widehat{\bB }^{\T}\widehat{\bF}_J(t_1)^{\T}\widehat{\bM}}\dd t_1 \dd t_2\\
&=\trace{\widetilde{\bcP}_{\p}^{\T}\bM\widetilde{\bcP}_{\p}\widehat{\bM}}.
\end{align*}
To show that $\widetilde{\bcP}_{\p}$ solves the Sylvester equation in \eqref{eq:PPSylv}, we follow the same argumentation as for \Cref{theo:LyQ_pp}.
\end{proof}

\Cref{lemma:Lintnorm_pp} and \ref{lemma:traceLintnorm_pp} result in the following theorem.
\begin{theorem}
We consider the asymptotically stable DAE system with a quadratic output equation from \eqref{eq:DAE_q}, the reduced-order model in \eqref{eq:redSyst}, the corresponding proper controllability Gramian $\bcP_{\p}$ as defined in \eqref{eq:Gramian_contr}, the reduced proper controllability Gramian $\widehat{\bcP}_{\p}$, and $\widetilde{\bcP}_{\p}$ as defined in \Cref{lemma:traceLintnorm_pp}.
The error between the proper-proper output $\by_{\p\p}(t)$ of the full-order model \eqref{eq:DAE_q} and the reduced output $\widehat{\by}_{\p\p}(t)$ satisfies the following bound:
\[
\|\by_{\p\p}-\widehat{\by}_{\p\p}\|_{L_\infty}^2 \leq \left(\trace{\bcP_{\p}\bM\bcP_{\p}\bM}-2\trace{\widetilde{\bcP}_{\p}^{\T}\bM\widetilde{\bcP}_{\p}\widehat{\bM}}+\trace{\widehat{\bcP}_{\p}\widehat{\bM}\widehat{\bcP}_{\p}\widehat{\bM}}\right)\|\bu\otimes \bu\|_{L_2}.
\]
\end{theorem}

\subsection{The improper-proper output error}
We want to estimate the improper-proper output error, i.e., the error between the improper-proper output $\by_{\ii\p}(t)$ and the reduced improper-proper output $\widehat{\by}_{\ii\p}(t)$. We define
\begin{equation}\label{eq:Hip}
\bh_{\ii\p}(t, k):= \bB ^{\T} \bF_{N}(k)^{\T}\bM \bF_{J}(t)\bB  \quad \text{ and }\quad
\widehat{\bh}_{\ii\p}(t, k):= \widehat{\bB }^{\T} \widehat{\bF}_{N}(k)^{\T}\widehat{\bM}\widehat{\bF}_{J}(t)\widehat{\bB }
\end{equation}
to obtain the output representations
\begin{align*}
\by_{\ii\p}(t) &=  \bx_{\ii}(t)^{\T}\bM \bx_{\p}(t)
=\int_{0}^{t}\sum_{k=0}^{\nu-1}\bh_{\ii\p}(t-\tau, k)
(\bu(\tau)\otimes \bu^{(k)}(t))\dd\tau,\\
\widehat{\by}_{\ii\p}(t) &=  \widehat{\bx}_{\ii}(t)^{\T}\widehat{\bM} \widehat{\bx}_{\p}(t) 
=\int_{0}^{t}\sum_{k=0}^{\nu-1}\widehat{\bh}_{\ii\p}(t-\tau, k)(\bu(\tau)\otimes \bu^{(k)}(t))\dd\tau.
\end{align*}
This representation of the improper-proper output can be used to obtain a bound of the $L_{\infty}$-error.

\begin{lemma}\label{lemma:Lintnorm_ip}
We consider the DAE system with a quadratic output equation from \eqref{eq:DAE_q}, the reduced-order model in \eqref{eq:redSyst}, and $\bh_{\ii\p}(t, k)$, $\widehat{\bh}_{\ii\p}(t, k)$ as defined in \eqref{eq:Hip}.
Then, the following bound holds
\begin{align*}
\|\by_{\ii\p}-\widehat{\by}_{\ii\p}\|_{L_{\infty}}\leq \left(\int_{0}^{\infty}\sum_{k=0}^{\nu-1} \| \bh_{\ii\p}(t, k) - \widehat{\bh}_{\ii\p}(t, k)\|_2^2 \dd \tau\right)^{\frac{1}{2}}\left(\int_{0}^{\infty}\sum_{k=0}^{\nu-1}\| \bu(\tau)\otimes \bu^{(k)}(t)\|^2_2 \dd \tau\right)^{\frac{1}{2}}.
\end{align*}
\end{lemma}
\begin{proof}
Using the definition \eqref{eq:Hip}, we obtain 
\begin{align*}
\big|\by_{\ii\p}(t)-\widehat{\by}_{\ii\p}(t)\big| = \bigg|\int_{0}^{t}\sum_{k=0}^{\nu-1} \left(\bh_{\ii\p}(t-\tau, k) - \widehat{\bh}_{\ii\p}(t-\tau, k)\right)\left(\bu(\tau)\otimes \bu^{(k)}(t)\right)\dd \tau \bigg|.
\end{align*}
By applying the Cauchy-Schwarz inequality multiple times, we obtain the following estimations
\begin{align*}
\big|\by_{\ii\p}(t)-\widehat{\by}_{\ii\p}(t)\big| &\leq \int_{0}^{t}\bigg|\sum_{k=0}^{\nu-1} \left(\bh_{\ii\p}(t-\tau, k) - \widehat{\bh}_{\ii\p}(t-\tau, k)\right)\left(\bu(\tau)\otimes \bu^{(k)}(t)\right) \bigg|\dd \tau\\
&\leq \int_{0}^{t}\left(\sum_{k=0}^{\nu-1} \| \bh_{\ii\p}(t-\tau, k) - \widehat{\bh}_{\ii\p}(t-\tau, k)\|_2^2 \right)^{\frac{1}{2}}\left(\sum_{k=0}^{\nu-1}\| \bu(\tau)\otimes \bu^{(k)}(t)\|^2_2 \right)^{\frac{1}{2}}\dd \tau\\
&\leq \left(\int_{0}^{t}\sum_{k=0}^{\nu-1} \| \bh_{\ii\p}(t, k) - \widehat{\bh}_{\ii\p}(t, k)\|_2^2 \dd \tau\right)^{\frac{1}{2}}\left(\int_{0}^{t}\sum_{k=0}^{\nu-1}\| \bu(\tau)\otimes \bu^{(k)}(t)\|^2_2 \dd \tau\right)^{\frac{1}{2}}.
\end{align*}
such that the $L_{\infty}$-norm of the output error is bounded by 
\begin{align*}
\|\by_{\ii\p}-\widehat{\by}_{\ii\p}\|_{L_{\infty}}\leq \left(\int_{0}^{\infty}\sum_{k=0}^{\nu-1} \| \bh_{\ii\p}(t, k) - \widehat{\bh}_{\ii\p}(t, k)\|_2^2 \dd \tau\right)^{\frac{1}{2}}\left(\int_{0}^{\infty}\sum_{k=0}^{\nu-1}\| \bu(\tau)\otimes \bu^{(k)}(t)\|^2_2 \dd \tau\right)^{\frac{1}{2}}.
\end{align*}
%with 
%\begin{align*}
%\int_{0}^{\infty}\sum_{k=0}^{\nu-1}\|\bh_{\ii\p}(t,k)-\widehat{\bh}_{\ii\p}(t, k)\|_2^2\dd t 
%=\int_{0}^{\infty}\sum_{k=0}^{\nu-1} \|\bh_{\ii\p}(t, k)\|_2^2  -2\langle \bh_{\ii\p}(t, k),\widehat{\bh}_{\ii\p}(t, k)\rangle  + \|\widehat{\bh}_{\ii\p}(t, k)\|_2^2\dd t.
%\end{align*}
\end{proof}

\begin{lemma}\label{lemma:traceLintnorm_ip}
We consider the asymptotically stable DAE system with a quadratic output equation from \eqref{eq:DAE_q}, the reduced-order model in \eqref{eq:redSyst}, the corresponding proper and improper controllability Gramian $\bcP_{\p}$ and $\bcP_{\ii}$ as defined in \eqref{eq:Gramian_contr}, and the reduced proper and improper controllability Gramians
\[
\widehat{\bcP}_{\p}:= \int_0^{\infty}\begin{bmatrix}
e^{\widehat{\bA}_1t}\widehat{\bB}_1\widehat{\bB}_1^{\T}e^{\widehat{\bA}_1^{\T}t} & 0 \\ 0 & 0
\end{bmatrix}\dd t, \qquad \widehat{\bcP}_{\ii}:= \sum_{k=0}^{\nu-1}\begin{bmatrix}
0 & 0 \\ 0 & \widetilde{\bE}_2^k\widetilde{\bB}_2\widetilde{\bB}_2^{\T}\left(\widetilde{\bE}_2^k\right)^{\T}
\end{bmatrix}.
\] The functionals $\bh_{\ii\p}(t, k)$ and $\widehat{\bh}_{\ii\p}(t, k)$ are as defined in \eqref{eq:Hip}.
Then the following equalities hold
\begin{align}
\int_{0}^{\infty}\sum_{k=0}^{\nu-1}\|\bh_{\ii\p}(t,k)\|_2^2\dd t &= \trace{\bcP_{\ii}\bM\bcP_{\p}\bM},\label{eq:PIeq1}\\
\int_{0}^{\infty} \sum_{k=0}^{\nu-1}\|\widehat{\bh}_{\ii\p}(t, k)\|_2^2\dd t_1 \dd t_2 
&= \trace{\widehat{\bcP}_{\ii}\widehat{\bM}\widehat{\bcP}_{\p}\widehat{\bM}},\label{eq:PIeq2}\\
\int_{0}^{\infty}\sum_{k=0}^{\nu-1}\langle \bh_{\ii\p}(t, k),\widehat{\bh}_{\ii\p}(t, k)\rangle\dd t &= \trace{\widetilde{\bcP}_{\ii}^{\T}\bM\widetilde{\bcP}_{\p}\widehat{\bM}}\label{eq:PIeq3}
\end{align}
where $\widetilde{\bcP}_{\p}$ is as in \Cref{lemma:traceLintnorm_pp} and $\widetilde{\bcP}_{\ii}:= \sum_{k=0}^{\nu-1}\bF_N(k)\bB \widehat{\bB }^{\T}\widehat{\bF}_N(k)^{\T}$ satisfies the projected Sylvester equation 
\begin{align}\label{eq:PISylv}
\bA\widetilde{\bcP}_{\ii} \widehat{\bA}^{\T} - \bE\widetilde{\bcP}_{\ii}\widehat{\bE}^{\T} = (\bI- \bP_{\Ll} )\bB \widehat{\bB }^{\T}(\bI-\widehat{\bP}_{\Ll}^{\T}), \qquad 0 = \bP_{\rr} \widetilde{\bcP}_{\ii} \widehat{\bP}_{\rr}^{\T}
\end{align}
with $\widehat{\bP}_{\Ll} = \widehat{\bP}_{\rr} = \begin{bmatrix}
	\bI_{\rr} & 0 \\ 0 & 0
\end{bmatrix}$, $\bP_{\Ll}$ and $\bP_{\rr}$ as defined in \eqref{eq:Proj}, and $\langle\cdot,\cdot\rangle$ denotes the Euclidean inner product with $\langle v_1,v_2\rangle = v_1^{\T}v_2$ for $v_1,v_2\in \R^{n}$.
\end{lemma}
\begin{proof}
The proof is analogous to the one from \Cref{lemma:traceLintnorm_pp}.
\end{proof}

\begin{theorem}\label{theo:ip_out}
For $\bu\in\C^{\nu-1}([0,\infty), \mathbb{R}^m)$ it holds
\begin{align*}
\int_{0}^{t}\sum_{k=0}^{\nu-1}\| \bu(\tau)\otimes \bu^{(k)}(t)\|^2_2 \dd \tau
\leq \nu \|\bu\|_{\bcC^{\nu-1}}^2\|\bu\|_{L_2}^2.
\end{align*}
\end{theorem}
\begin{proof}
Applying Kronecker product properties and Cauchy-Schwarz inequality yields 
\begin{align*}
\int_{0}^{t}\sum_{k=0}^{\nu-1}\| \bu(\tau)\otimes \bu^{(k)}(t)\|^2_2 \dd \tau
&= \int_{0}^{t}\sum_{k=0}^{\nu-1}(\bu^{(k)}(t)\otimes \bu(\tau)) ^{\T}(\bu(\tau)\otimes \bu^{(k)}(t)) \dd \tau\\
&= \int_{0}^{t}\sum_{k=0}^{\nu-1}(\bu^{(k)}(t)^{\T}\otimes \bu(\tau)^{\T}) (\bu(\tau)\otimes \bu^{(k)}(t)) \dd \tau\\
&= \int_{0}^{t}\sum_{k=0}^{\nu-1} ( \bu^{(k)}(t)^{\T} \bu(\tau)) \otimes  ( \bu(\tau)^{\T}\bu^{(k)}(t) ) \dd \tau\\
&= \int_{0}^{t}\sum_{k=0}^{\nu-1} \bu^{(k)}(t)^{\T} \bu(\tau)\bu(\tau)^{\T}\bu^{(k)}(t) \dd \tau\\
&\leq \sum_{k=0}^{\nu-1} \int_{0}^{\infty} \| \bu(\tau)\|^2 \dd \tau \|\bu^{(k)}(t)\|^2\\
&= \sum_{k=0}^{\nu-1} \| \bu(\tau)\|_{L_2}^2 \|\bu^{(k)}(t)\|^2
\leq \nu \|\bu\|_{\bcC^{\nu-1}}^2\|\bu\|_{L_2}^2
\end{align*}
for $\|\bu\|_{\bcC^{\nu-1}}:=\max_{k = 0, \dots, \nu-1}\sup_{t\geq 0} \|\bu\|_2$.
\end{proof}

Together with \Cref{theo:ip_out},  \Cref{lemma:Lintnorm_ip} and \Cref{lemma:traceLintnorm_ip} we obtain the following theorem.
\begin{theorem}
We consider the asymptotically stable DAE system with a quadratic output equation \eqref{eq:DAE_q}, the reduced-order model in \eqref{eq:redSyst}, the corresponding proper and improper controllability Gramians $\bcP_{\p}$, $\bcP_{\ii}$ as defined in \eqref{eq:Gramian_contr}, the reduced proper and improper controllability Gramians $\widehat{\bcP}_{\p}$, $\widehat{\bcP}_{\ii}$, and $\widetilde{\bcP}_{\p}$, $\widetilde{\bcP}_{\ii}$ as defined in Lemma \ref{lemma:traceLintnorm_ip}.
The error between the improper-proper output $\by_{\ii\p}$ of the full-order model \eqref{eq:DAE_q} and the reduced output $\widehat{\by}_{\ii\p}$ satisfies the following bound:
\[\|\by_{\ii\p}-\widehat{\by}_{\ii\p}\|_{L_\infty}\leq \Biggl(
\trace{\bcP_{\p}\bM\bcP_{\ii}\bM} 
-2\trace{\widetilde{\bcP}_{\p}^{\T}\bM\widetilde{\bcP}_{\ii}\widehat{\bM}}+\trace{\widehat{\bcP}_{\p}\widehat{\bM}\widehat{\bcP}_{\ii}\widehat{\bM}}
\Biggr)^{\frac{1}{2}} \nu ^{\frac{1}{2}} \|\bu\|_{\bcC^{\nu-1}}\|\bu\|_{L_2}\]
for output functions $\bu\in\C^{\nu-1}([0,\infty), \mathbb{R}^m)$.
\end{theorem}

Since $\|\by_{\p\ii}-\widehat{\by}_{\p\ii}\|_{L_\infty}^2$ is equal to $\|\by_{\ii\p}-\widehat{\by}_{\ii\p}\|_{L_\infty}^2$, and the improper states are not truncated
the overall error $\|\by-\widehat{\by}\|_{L_{\infty}}$ can be estimated as 
\begin{align}\label{eq:finalErrBound}
\|\by-\widehat{\by}\|_{L_{\infty}} &\leq \|\by_{\p\p}-\widehat{\by}_{\p\p}\|_{L_{\infty}} + 2\cdot\|\by_{\p\ii}-\widehat{\by}_{\p\ii}\|_{L_{\infty}}\nonumber\\
&\leq \Biggl(\trace{\bcP_{\p}\bM\bcP_{\p}\bM}-2\trace{\widetilde{\bcP}_{\p}^{\T}\bM\widetilde{\bcP}_{\p}\widehat{\bM}}+\trace{\widehat{\bcP}_{\p}\widehat{\bM}\widehat{\bcP}_{\p}\widehat{\bM}}\Biggr)^{\frac{1}{2}}\|\bu\otimes \bu\|_{L_2}^{\frac{1}{2}}\\
&\qquad + 
2\cdot\Biggl(
\trace{\bcP_{\p}\bM\bcP_{\ii}\bM} 
-2\trace{\widetilde{\bcP}_{\p}^{\T}\bM\widetilde{\bcP}_{\ii}\widehat{\bM}}+\trace{\widehat{\bcP}_{\p}\widehat{\bM}\widehat{\bcP}_{\ii}\widehat{\bM}}
\Biggr)^{\frac{1}{2}} \nu ^{\frac{1}{2}} \|\bu\|_{\bcC^{\nu-1}}\|\bu\|_{L_2}.\nonumber
\end{align}

\section{Extension to the multiple output case}\label{sec:mulOutGram} 
Up to now, we considered systems \eqref{eq:DAE_q} with a single output.
In this section, however, we will extend the previous theory to the multiple output case, where the output is given as 
\begin{align}
\by(t) = \bC\bx(t) + \begin{bmatrix}
\bx(t)^{\T}\bM_1\bx(t)\\
\vdots \\
\bx(t)^{\T}\bM_p\bx(t)
\end{bmatrix}
\end{align}
where $\bC\in\mathbb{R}^{p\times n}$ and $\bM_k = \bM_k^{\T}\in\Rnn$ for all $k=1, \dots, p$.
As we already did for the DAE system with one quadratic output \eqref{eq:DAE_q} we consider the different parts of the output separately so that we investigate 
\[
\by_{C}(t):= \bC\bx(t),\qquad \by_{j}(t):= \bx(t)^{\T}\bM_j\bx(t),\quad j=1,\dots, p
\]
and derive the corresponding Gramians that are summed up in the end to derive Gramians that cover the overall observability behavior.

For the linear term $\by_{C}(t)$, define the proper and improper observability mapping 
\[
\bcC^C_{\p}(t):=\bC\bF_J(t), \qquad \bcC^C_{\ii}(k):=\bC\bF_N(k)
\]
and apply the theory from \cite{morSty04} to obtain the proper and improper observability Gramian 
\begin{align*}
\bcQ_{\p}^{C}&:=\int_{0}^{\infty} \left(\bcC^C_{\p}(t)\right)^{\T}\bcC^C_{\p}(t)\dd t:=\int_{0}^{\infty} \bF_J(t)^{\T}\bC^{\T}\bC \bF_J(t)\dd t,\\ \bcQ_{\ii}^{C}&:=\sum_{k=0}^{\nu-1} \left(\bcC^C_{\ii}(k)\right)^{\T}\bcC^C_{\ii}(k):=\sum_{k=0}^{\nu-1} \bF_N(k)^{\T}\bC^{\T}\bC \bF_N(k).
\end{align*}

For each of the quadratic components, we define the observability mappings 
\begin{align*}
&\bcO_{\p\p}^j(t_1, t_2) := \bB ^{\T}\bF_J(t_1)^{\T}\bM_j \bF_J(t_2),\quad
\bcO_{\p\ii}^j(t, k) := \bB ^{\T}\bF_J(t)^{\T}\bM_j \bF_N(k),\\
&\bcO_{\ii\p}^j(t, k) := \bB ^{\T}\bF_N(k)^{\T}\bM_j \bF_J(t),\qquad\;\,
\bcO_{\ii\ii}^j(\ell, k) := \bB ^{\T}\bF_N(\ell)^{\T}\bM_j \bF_J(t)
\end{align*}
%for $j = 1, \dots, p$, 
and apply the theory from Section \ref{sec:Gramians} to derive the corresponding observability Gramians
\[
\bcQ_{\p}^{j}:=\int_{0}^{\infty} \bF_J(t)^{\T}\bM_j(\bcP_{\p}+\bcP_{\ii})\bM_j\bF_J(t)\dd t,\quad \bcQ_{\ii}^{j}:=\sum_{k=0}^{\nu-1} \bF_N(k)^{\T}\bM_j(\bcP_{\p}+\bcP_{\ii})\bM_j\bF_N(k)
\]
for $j = 1,\dots, p.$
Finally, we can describe the overall observability behavior using the proper and improper observability Gramians, which we define as
\begin{align}\label{eq:multOut}
\bcQ_{\p}:= \bcQ_{\p}^{C} + \sum_{j=1}^{p} \bcQ_{\p}^{j}, \qquad \bcQ_{\ii}:= \bcQ_{\ii}^{C} + \sum_{j=1}^{p} \bcQ_{\ii}^{j}.
\end{align}

To describe the output energies, we first gather the output mappings defined above to a general proper and improper observability mapping
\[
\bcO_{\p}(k, t_1, t_2):=\begin{bmatrix}
\bcQ_{\p}^{C}(t_1) \\[5pt]
\bcO_{\p\p}^1(t_1, t_2)\\ \vdots \\ 
\bcO_{\p\p}^p(t_1, t_2)\\[5pt]
\bcO^1_{\ii\p}(k, t_2)\\ \vdots \\
\bcO_{\ii\p}^p(k, t_2)
\end{bmatrix} \qquad\text{ and }\qquad
\bcO_{\ii}(\ell, k, t):=\begin{bmatrix}
\bcQ_{\ii}^{C}(\ell) \\[5pt]
\bcO_{\p\ii}^1(\ell, t)\\ \vdots \\ 
\bcO_{\p\ii}^p(\ell, t)\\[5pt]
\bcO^1_{\ii\ii}(\ell, k)\\ \vdots \\
\bcO_{\ii\ii}^p(\ell, k)
\end{bmatrix}.
\]
Together with the derivations from \Cref{ssec:Reduction_input} we obtain the proper and improper output energies
\begin{align*}
E(\bcO_{\p}) = \|\bcO_{\p}\|^2 = \trace{\bcQ_{\p}},\quad
E(\bcO_{\ii}) = \|\bcO_{\ii}\|^2 = \trace{\bcQ_{\ii}},
\end{align*}
where $\bcQ_{\p}$ and $\bcQ_{\ii}$  are as defined in \eqref{eq:multOut}.
These energy expressions justify the truncation process as described in \Cref{ssec:BT} also for the multiple output case.

To estimate the error $\|\by-\widehat{\by}\|_{L_{\infty}}$ in the case of multiple outputs, we estimate the output norm by the sum of the norms of the different components of the output, that is
\begin{align*}
\|\by-\widehat{\by}\|_{L_{\infty}}&\leq \|\by_C-\widehat{\by}_C\|_{L_{\infty}} +  \|\begin{bmatrix}
\by_1-\widehat{\by}_1\\
\vdots\\
\by_{\p}-\widehat{\by}_{\p}
\end{bmatrix}\|_{L_{\infty}}\\
&\leq \|\by_C-\widehat{\by}_C\|_{L_{\infty}} + \|\begin{bmatrix}
\by_1-\widehat{\by}_1\\
0\\
\vdots
\end{bmatrix}\|_{L_{\infty}}+
\dots +  \|\begin{bmatrix}
 \vdots\\
 0\\
 \by_{\p}-\widehat{\by}_{\p}
 \end{bmatrix}\|_{L_{\infty}}\\
 &\leq \|\by_C-\widehat{\by}_C\|_{L_{\infty}} + \|
 \by_1-\widehat{\by}_1
\|_{L_{\infty}}+
 \dots +  \|
\by_{\p}-\widehat{\by}_{\p}
\|_{L_{\infty}}.
\end{align*}
In the first component $\|\by_C-\widehat{\by}_C\|_{L_{\infty}}$, the improper part can be neglected since the improper states are not truncated.
Hence, it holds 
\begin{align*}
\|\by_C-\widehat{\by}_C\|_{L_{\infty}} 
&= \|\bC\bx_{\p}-\widehat{\bC}\widehat{\bx}_{\p}\|_{L_{\infty}}
= \left(\int_{0}^{\infty}\|\mathrm{vec}(\bC\bF_J(t)\bB -\widehat{\bC}\widehat{\bF}_J(t)\widehat{\bB })\|_2^2\dd t\right)^{\frac{1}{2}}
\left(\int_{0}^{\infty}\|\bu(\tau)\|_2^2\dd t\right)^{\frac{1}{2}}\\
&= \mathrm{tr}\left( \bB ^{\T}\bcQ_{\p}\bB \right)\|\bu\|_{L_2}
\end{align*}
For the other summands, we apply the theory presented in \Cref{sec:ErrEst} to obtain the bound
\begin{align*}
\sum_{j=1}^p \|
\by_j-&\widehat{\by}_j
\|_{L_{\infty}}\\
&\leq 
\sum_{j=1}^p \|\by_{\p\p,j}-\widehat{\by}_{\p\p,j}\|_{L_{\infty}} +  \|\by_{\ii\p,j}-\widehat{\by}_{\ii\p,j}\|_{L_{\infty}}+  \|\by_{\p\ii,j}-\widehat{\by}_{\p\ii,j}\|_{L_{\infty}}\\
&=
\sum_{j=1}^p 
\left(\trace{\bcP_{\p}\bM_j\bcP_{\p}\bM_j}-2\trace{\widetilde{\bcP}_{\p}^{\T}\bM_j\widetilde{\bcP_{\p}}\widehat{\bM}_j}+\trace{\widehat{\bcP}_{\p}\widehat{\bM}_j\widehat{\bcP}_{\p}\widehat{\bM}_j}\right)\|\bu\otimes \bu\|_{L_2}\\
&\quad +\Biggl(
\trace{\bcP_{\p}\bM_j\bcP_{\ii}\bM_j} 
-2\trace{\widetilde{\bcP}_{\p}^{\T}\bM_j\widetilde{\bcP}_{\ii}\widehat{\bM}_j}+\trace{\widehat{\bcP}_{\p}\widehat{\bM}_j\widehat{\bcP}_{\ii}\widehat{\bM}_j}
\Biggr)^{\frac{1}{2}} \nu ^{\frac{1}{2}} \|\bu\|_{\bcC^{\nu-1}}\|\bu\|_{L_2}\\
&\quad+ \Biggl(
\trace{\bcP_{\p}\bM_j\bcP_{\ii}\bM_j} 
-2\trace{\widetilde{\bcP}_{\p}^{\T}\bM_j\widetilde{\bcP}_{\ii}\widehat{\bM}_j}+\trace{\widehat{\bcP}_{\p}\widehat{\bM}_j\widehat{\bcP}_{\ii}\widehat{\bM}_j}
\Biggr)^{\frac{1}{2}} \nu ^{\frac{1}{2}} \|\bu\|_{\bcC^{\nu-1}}\|\bu\|_{L_2}.
\end{align*}

\section{Numerical Results}\label{sec:NumRes}
In this section, we discuss the efficiency of the proposed methodology using several examples. 
We also verify our theoretical findings, particularly the error bounds, in our numerical experiments. 
All the numerical experiments are carried out on a computer with 4 Intel Core i5-4690 CPUs running at 3.5 GHz and equipped with 8 GB total main memory. 
The experiments use \matlab R2017a and examples and methods from M-M.E.S.S.-2.1., see \cite{SaaKB21-mmess-2.1}.

\subsection{An illustrative example}

Using this example, we highlight that for quadratic output models, it is necessary to consider mixed Gramians  $\bcQ_{\p\ii}$ and $\bcQ_{\ii\p}$, as discussed in \Cref{sec:Gramians}. 
%First, we show the need of the mixed Gramians $\cQ_{\p\ii}$ and $\cQ_{\ii\p}$ for a simple illustrative example.
For this, we consider the following system in Weierstra{\ss} canonical form 
\begin{align*}
\begin{bmatrix}
1 & 0 & 0 & 0 \\
0 & 1 & 0 & 0 \\
0 & 0 & 0 & 1 \\
0 & 0 & 0 & 0 \\
\end{bmatrix}\begin{bmatrix}
\dot{z}_1(t) \\
\dot{z}_2(t) \\
\dot{z}_3(t) \\
\dot{z}_4(t) \\
\end{bmatrix} &= \begin{bmatrix}
-1 & 0 & 0 & 0 \\
0 & -1 & 0 & 0 \\
0 & 0 & 1 & 0 \\
0 & 0 & 0 & 1 \\
\end{bmatrix}\begin{bmatrix}
z_1(t) \\
z_2(t) \\
z_3(t) \\
z_4(t) \\
\end{bmatrix}\begin{bmatrix}
1 \\
1 \\
1 \\
1 \\
\end{bmatrix}\bu(t), \\ 
\by(t) &= \begin{bmatrix}
z_1(t) &
z_2(t) &
z_3(t) &
z_4(t)
\end{bmatrix}\begin{bmatrix}
1 & 0 & 1 & 0 \\
0 & 0 & 0 & 1 \\
1 & 0 & 0 & 0 \\
0 & 1 & 0 & 2 \\
\end{bmatrix}\begin{bmatrix}
z_1(t) \\
z_2(t) \\
z_3(t) \\
z_4(t) \\
\end{bmatrix}.
\end{align*}
The proper state is then given by $\bx_1(t) = \begin{bmatrix}
z_1(t)\\ z_2(t)
\end{bmatrix}$ and the improper one as $\bx_2(t) = \begin{bmatrix}
z_3(t)\\ z_4(t)
\end{bmatrix}$.
The corresponding system Gramians are given as 
\begin{align*}
\bcP_1 = \begin{bmatrix}
\frac{1}{2} & \frac{1}{2} \\ \frac{1}{2} & \frac{1}{2}
\end{bmatrix},\; \;
\bcP_2 = \begin{bmatrix}
2 & 1 \\ 1 & 1
\end{bmatrix}, \; \;
\bcQ_{11} = \begin{bmatrix}
\frac{1}{4} & 0 \\ 0 & 0
\end{bmatrix},\; \;
\bcQ_{21} = \begin{bmatrix}
1 & \frac{1}{2} \\ \frac{1}{2} & \frac{1}{2}
\end{bmatrix},\; \;
\bcQ_{12} = \begin{bmatrix}
\frac{1}{2} & \frac{1}{2} \\ \frac{1}{2} & 1
\end{bmatrix},\;  \;
\bcQ_{22} = \begin{bmatrix}
0 & 0 \\ 0 & 4
\end{bmatrix}.
\end{align*}
We observe that the controllability Gramian of the proper state is of rank one. 
Hence, the minimal realization of the proper part of the system is of rank one, and so is the proper part of the reduced-order model for this example.
The improper state is described by a rank two controllability Gramian and by a rank two observability Gramian that is $\bcQ_{\ii} = \bcQ_{\p\ii} + \bcQ_{\ii\ii}$ so that the minimal realization of the improper system part is of rank two.
However, we observe that the improper-improper observability Gramian is of rank one. 
This fact shows vividly that the mixed Gramians need to be taken into consideration. 

To investigate the quality of the obtained reduced-order system, we consider the system output, which we obtain by applying the input function $\bu(t) =  0.2\cdot e^{-t}$. 
The results are shown in \Cref{fig:WCF}, where the left plot shows the results of the full-order model (\texttt{FOM}), the reduced-order model (\texttt{ROM}), and the corresponding error (\texttt{Error}) when the mixed Gramians are applied in the reduction process.
The right plot shows the same values for the case when the mixed Gramians were not part of the reduction step, i.e., $\bcQ_{\p}:=\bcQ_{\p\p}$ and $\bcQ_{\ii}:=\bcQ_{\ii\ii}$.
\begin{figure}[tb]
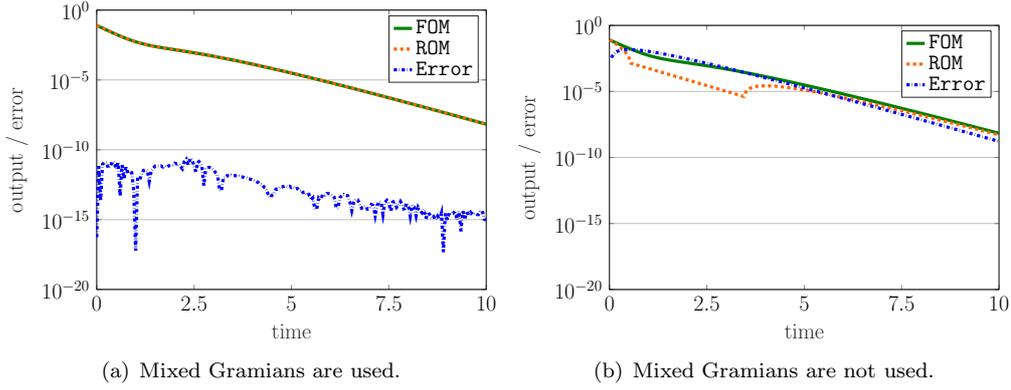

\subfigure[Mixed Gramians are used.]{\input{Index1}}
\subfigure[Mixed Gramians are not used.]{\input{Index1false}}
\caption{An illustrative example: Output responses of the full-order and reduced-order models and the corresponding error.}\label{fig:WCF}
\end{figure}
We note that the mixed observability Gramians $\bcQ_{\p\ii}$ and $\bcQ_{\ii\p}$ must be considered within the reduction process.

\subsection{Index--2 Stokes example}

As the second example, we consider the creeping flow in capillaries or porous media, which can be described by the following equations
\begin{align}\label{Stokes}
\begin{split}
\frac{\mathrm{d}}{\mathrm{d}t}v(\zeta,t) &= \mu\Delta v(\zeta,t) - \nabla p(\zeta,t) +f(\zeta,t),\\
0 &= \mathrm{div}(v(\zeta,t)),
\end{split}
\end{align}
with appropriate initial and boundary conditions. 
The position in the domain $\Omega\subset\mathbb{R}^{d}$ is described by $\zeta\in\Omega$ and $t\geq0$ is the time.
For simplicity, we use a classical solution concept and assume that the external force $f: \Omega \times [0,\infty) \to \mathbb{R}^d$ is continuous and that 
the velocities $v: \Omega \times [0,\infty) \to \mathbb{R}^d$ and pressures $p: \Omega \times [0,\infty) \to \mathbb{R}^d$ satisfy the necessary smoothness conditions.
We discretize the system \eqref{Stokes} by a finite difference scheme as discussed in \cite{morMehS05, morSty06} and have an output equation to measure our quantity of interest.
We choose the matrix $\bM$ to be $0.01\cdot \bI_n$, yielding the $l_2$-norm of the state vector with a scaling factor $0.01$. Consequently, we obtain a discretized system of the form 
\begin{align}\label{Index2}
\begin{split}
\frac{\mathrm{d}}{\mathrm{d}t}\begin{bmatrix}
I & 0 \\
0 & 0
\end{bmatrix}\begin{bmatrix}
z(t)\\
{\lambda}(t)
\end{bmatrix}&=\begin{bmatrix}
A & G \\
G^{\mathrm{T}} & 0
\end{bmatrix}\begin{bmatrix}
z(t) \\
\lambda(t)
\end{bmatrix} + \begin{bmatrix}
B _1\\
B _2
\end{bmatrix}\bu(t),\\
\by(t) &= \begin{bmatrix}
z(t)^{\T} &
\lambda(t)^{\T}
\end{bmatrix}\bM\begin{bmatrix}
z(t) \\
\lambda(t)
\end{bmatrix}
\end{split}
\end{align}
with system matrices $A\in\mathbb{R}^{g\times g}$ and $G\in\mathbb{R}^{g, q}$.
The input matrices are given as $B _1\in\mathbb{R}^{g, m}$, $B _2\in\mathbb{R}^{q, m}$ and the output matrix is $\bM\in\mathbb{R}^{n, n}$ with $n = g+q$.
The state consists of $z(t)\in\mathbb{R}^{g}$ and $\lambda(t)\in\mathbb{R}^{q}$, while the input is $\bu(t)\in\mathbb{R}^{m}$ and the output $\by(t)\in\mathbb{R}$.
We consider the system of dimension $n = 645 = n_v + n_p$, where the dimensions of the velocity and pressure vectors are $n_v = 420$ and $n_p=225$, respectively.

We need to determine the Gramians corresponding to the proper and improper states of the system \eqref{Index2}.
For this purpose, we apply the methods described in \cite{Sty08,morSty06}, noting that the improper Gramians can be computed explicitly.
In Figure \ref{fig:Index2SVD}, we depict the decay of the Hankel singular values $\sigma_1,\, \sigma_2, \dots$ of the proper states corresponding to the proper Gramians as described in Section \ref{ssec:BT}.
We truncate the proper Hankel singular values smaller than $\sigma_1\cdot 10^{-8}$ and truncate the improper Hankel singular values equal to zero.
The reduced-order model has the dimensions $\widehat{n}=\widehat{n}_v + \widehat{n}_p$ with $n_v = 13$ and $\hat{n}_p=2$.
\Cref{fig:Index2} shows the output behavior of the full-order model \eqref{eq:DAE_q} and of the reduced-order model \eqref{eq:redDAE_q} for an input function $\bu(t) =  \mathrm{sin}(t)^3e^{-t/2}$. 
Additionally, the figure includes the output error and the corresponding error estimation.
The actual error is below the estimated error for all time, and we observe that the error bound is rather conservative.
The correct error is sufficiently small, and the approximation quality of the reduced-order systems is much better than the estimated one. 

\begin{figure}
\definecolor{mycolor1}{rgb}{0.87059,0.49020,0.00000}%
\begin{tikzpicture}[scale=0.3]
\huge
\begin{axis}[%
width=7.589in,
height=3.761in,
at={(1.273in,0.575in)},
scale only axis,
xmin=0,
xmax=30,
xtick={0,5, 10, 15, 20, 25, 30},
xlabel style={font=\color{white!15!black}},
xlabel={$k$},
ymode=log,
ymin=1e-15,
ymax=1,
ytick={1e-15, 1e-10, 1e-5,1e-0},
yminorticks=true,
ylabel style={font=\color{white!15!black}},
ylabel={Hankel singular values $\sigma_k$},
axis background/.style={fill=white},
xmajorgrids,
ymajorgrids,
yminorgrids,
legend style={legend cell align=left, align=left, draw=white!15!black}
]
\addplot [color=mycolor1, line width=4.0pt, mark=diamond, mark options={solid, mycolor1}, forget plot]
  table[row sep=crcr]{%
1	0.00958895243501904\\
2	0.00517510300220882\\
3	0.00102975834886902\\
4	0.000923394365001987\\
5	0.000179434611337845\\
6	0.000113004730796571\\
7	2.43937802112031e-05\\
8	1.6658045903493e-05\\
9	3.21576181005274e-06\\
10	2.58589909581577e-06\\
11	5.01134228779968e-07\\
12	3.41594990957443e-07\\
13	6.75468308736651e-08\\
14	4.88666558496e-08\\
15	9.32052140101513e-09\\
16	7.20180774427551e-09\\
17	1.36521158204146e-09\\
18	9.74155478606235e-10\\
19	1.9317914523185e-10\\
20	1.34017597868412e-10\\
21	2.53052107625101e-11\\
22	1.99739851514888e-11\\
23	3.74669414660914e-12\\
24	2.62417261468859e-12\\
25	5.0077767957232e-13\\
26	3.557266623208e-13\\
27	5.7011867579097e-14\\
28	4.18178112891583e-14\\
29	4.77767045841365e-15\\
30	3.25850866049754e-15\\
};
\end{axis}
\end{tikzpicture}%
	\centering
	\caption{Index--2 Stokes example: the decay of Hankel singular values.}
	\label{fig:Index2SVD}
\end{figure}
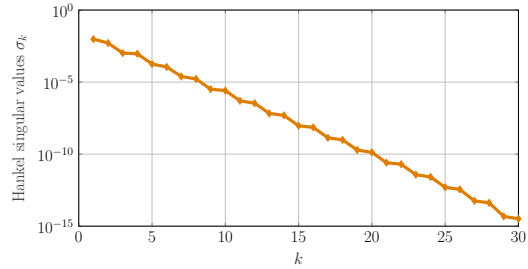

\begin{figure}[tb]
\subfigure[Output]{\input{Index2output}}
\subfigure[Error]{\input{Index2output_err}}
\caption{Index--2 Stokes example: the outputs of the full-order model and the obtained reduced-order model of order $15$ for a test input $\bu(t) = \mathrm{sin}(t)^3e^{-t/2}$.  
The plot also shows the error between the outputs and our error estimate. 
}
\label{fig:Index2}
\end{figure}

\subsection{Index--3 mechanical  system}

Now, we investigate a system of index--3, that results from mechanical systems and is of the form 
\begin{align}\label{Index3}
\begin{split}
\frac{\mathrm{d}}{\mathrm{d}t}\begin{bmatrix}
I_{n_x} & 0 & 0 \\
0 & H & 0 \\
0 & 0  & 0
\end{bmatrix}\begin{bmatrix}
x_1(t)\\
x_2(t)\\
\lambda(t)
\end{bmatrix} &= \begin{bmatrix}
0 & I_{n_x} & 0\\
-K & -D & G\\
G^{\mathrm{T}} & 0 & 0
\end{bmatrix}\begin{bmatrix}
x_1(t)\\
x_2(t)\\
\lambda(t)
\end{bmatrix} + \begin{bmatrix}
0\\
B _{x}\\
0
\end{bmatrix}\bu(t),\\
\by(t) &= \begin{bmatrix}
x_1(t)^{\T} & x_2(t)^{\T} & \lambda(t)^{\T}
\end{bmatrix}\bM\begin{bmatrix}
x_1(t)\\
x_2(t)\\
\lambda(t)
\end{bmatrix},
\end{split}
\end{align}
where $H,~D,~K\in\mathbb{R}^{g\times g}$, $B_x\in\mathbb{R}^{g\times m}$, $G\in\mathbb{R}^{g\times q}$ and $\bM\in\mathbb{R}^{2g+\ell\times 2n+q}$.
The state is given by $x_1(t),~x_2(t)\in\mathbb{R}^{g},~ \lambda(t)\in\mathbb{R}^{g}$, the input by $\bu(t)\in\mathbb{R}^{m}$ and the output by $\by(t)\in\mathbb{R}$.
We consider the index--3 system \eqref{Index3}, which arises in the modeling of constraint mechanical systems with matrices
\begin{align*}
H &= \mathrm{diag}(m_1,\, \ldots,\, m_g),\\
D &=\begin{bmatrix}
d_1 + \delta_1 & -d_1 &  &    \\
-d_1 &   d_1 + d_2+\delta_2 & \hspace{-40pt} -d_2  &     \\
&    \ddots &\ddots & \ddots    \\
&   -d_{g-2} &      d_{g-2}+d_{g-1}+\delta_{g-1}    & -d_{g-1} \\
&   &     -d_{g-1} &  d_{g-1} + \delta_g
\end{bmatrix},\\
K &= 
\begin{bmatrix}
k_1 + \kappa_1 & -k_1 &  &    \\
-k_1 &   k_1 + k_2+\kappa_2 & \hspace{-40pt} -k_2  &     \\
&    \ddots&\ddots & \ddots     \\
&   -k_{g-2} &      k_{g-2}+k_{g-1}+\kappa_{g-1}    & -k_{g-1} \\
&   &     -k_{g-1} &  k_{g-1} + \kappa_g
\end{bmatrix},\\
&
\\[5pt]
\quad G &= [1,\,0,\,\dots,\,0,\,-1]^{\mathrm{T}},\quad
B_x  = [1,\,0,\,\dots,\,0]^{\mathrm{T}}, \quad \bM = \bI_{2g+1}.
\end{align*}
The matrices are generated using the M-M.E.S.S.-function $\mathtt{msd\_ind3}$, see \cite{SaaKB21-mmess-2.1},  with dimension $ g = 600$. We choose 
\begin{align*}
m_1 &= \cdots = m_g = 1, \quad k_1 = \cdots = k_{g-1} = 1.5, \quad d_1 = \cdots = d_{g-1} = 0.7, \\
\kappa_1 &= \cdots =\kappa_g = 2 ,\quad \delta_1 = \cdots = \delta_g = 0.9.
\end{align*}
In \cite{morMehS05}, the projection matrices \eqref{eq:Proj} for this example were introduced.
To compute the Gramians, we follow the same procedure presented in \cite{morSty06,Sty08} modified to the index 3 cases.

\Cref{fig:Index3SVD} depicts the proper Hankel singular values. We truncate those smaller than $\sigma_1\cdot 10^{-8}$. 
Additionally, we remove the improper states corresponding to improper Hankel singular values that are zero.
The resulting reduced dimensions are $\widehat{n}=\widehat{n}_v + \widehat{n}_p$ with $n_v = 20$ and $\hat{n}_p=1$.
The outputs of the full-order and the reduced-order models \eqref{eq:DAE_q} and \eqref{eq:redDAE_q} are described in \Cref{fig:Index3} for an input function $\bu(t) =  \mathrm{sin}(2t)^2e^{-t/2}$ 
and the figure also shows the error between the outputs and the error estimate using \ref{eq:finalErrBound}. 
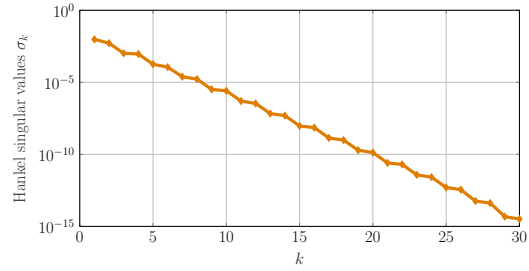
\begin{figure}
\definecolor{mycolor1}{rgb}{0.87059,0.49020,0.00000}%
\begin{tikzpicture}[scale=0.3]
\huge
\begin{axis}[%
width=7.589in,
height=3.761in,
at={(1.273in,0.575in)},
scale only axis,
xmin=0,
xmax=30,
xtick={0,5, 10, 15, 20, 25, 30},
xlabel style={font=\color{white!15!black}},
xlabel={$k$},
ymode=log,
ymin=1e-15,
ymax=1,
ytick={1e-15, 1e-10, 1e-5,1e-0},
yminorticks=true,
ylabel style={font=\color{white!15!black}},
ylabel={Hankel singular values $\sigma_k$},
axis background/.style={fill=white},
xmajorgrids,
ymajorgrids,
yminorgrids,
legend style={legend cell align=left, align=left, draw=white!15!black}
]
\addplot [color=mycolor1, line width=4.0pt, mark=diamond, mark options={solid, mycolor1}, forget plot]
  table[row sep=crcr]{%
1	0.00958895243501904\\
2	0.00517510300220882\\
3	0.00102975834886902\\
4	0.000923394365001987\\
5	0.000179434611337845\\
6	0.000113004730796571\\
7	2.43937802112031e-05\\
8	1.6658045903493e-05\\
9	3.21576181005274e-06\\
10	2.58589909581577e-06\\
11	5.01134228779968e-07\\
12	3.41594990957443e-07\\
13	6.75468308736651e-08\\
14	4.88666558496e-08\\
15	9.32052140101513e-09\\
16	7.20180774427551e-09\\
17	1.36521158204146e-09\\
18	9.74155478606235e-10\\
19	1.9317914523185e-10\\
20	1.34017597868412e-10\\
21	2.53052107625101e-11\\
22	1.99739851514888e-11\\
23	3.74669414660914e-12\\
24	2.62417261468859e-12\\
25	5.0077767957232e-13\\
26	3.557266623208e-13\\
27	5.7011867579097e-14\\
28	4.18178112891583e-14\\
29	4.77767045841365e-15\\
30	3.25850866049754e-15\\
};
\end{axis}
\end{tikzpicture}%
	\centering
	\caption{Index--3 Mechanical system example: the decay of Hankel singular values.}\label{fig:Index3SVD}
\end{figure}

\begin{figure}[tb]
\subfigure[Output]{\input{Index3output}}
\subfigure[Error]{\input{Index3output_err}}
\caption{Index--3 Mechanical system example: the outputs of the full-order model and the obtained reduced-order model of order $21$ for the input $\bu(t) =  \mathrm{sin}(2t)^2e^{-t/2}$.  
The plot also shows the error between the outputs and our error estimate. 
}\label{fig:Index3}
\end{figure}
Here again, we make similar observations as in the previous example--that is, the error bound is conservative, and the system is approximated much better, leading to an output error that is smaller than $10^{-13}$ for all $t\in[0,\, 10]$.

\section*{Conclusions}

This paper has addressed the model reduction problem for differential-algebraic systems with quadratic output equations using BT.
To apply the balancing procedure, we have derived new observability Gramians that describe the observability behavior of the proper and improper states.
We have shown that these Gramians can be computed by solving certain projected Lyapunov equations. 
Moreover, we have derived some bounds of the energy functional for the proper states that have been used to derive a truncation criterion.
To evaluate the quality of the reduced surrogate model, we have introduced an error estimator.

Our method has been illustrated by numerical examples of indexes one, two, and three. In particular, we were able to derive surrogate models of very small dimensions that approximate the input--to--output behavior of the full-order models very well. 
\bibliographystyle{plainurl} 
\bibliography{mor,References,csc,software}

\end{document}